\documentclass{siamart190516}[11pt]
\usepackage[a4paper
	]{geometry}
\usepackage{amssymb}
\usepackage{amsfonts}
\usepackage{amsmath,mathrsfs,amstext, array}
\usepackage{enumerate}
\usepackage{bbold,bm, color}
\usepackage{graphicx}
\usepackage{natbib}
\usepackage{algorithmic}
\usepackage{url}
\usepackage{multirow}
\newcommand{\dd}{{\rm d}}

\newcommand{\domain}{\mathbb{D}}
\newcommand{\To}{\mathbb{T}}

\newcommand{\ignore}[1]{}

\newcommand{\C}{\mathcal{C}}
\newcommand{\D}{\mathcal{K}}

\newcommand{\J}{\mathcal{G}}
\newcommand{\K}{\mathcal{H}}
\newcommand{\calL}{\mathcal{L}}
\renewcommand{\P}{\mathcal{P}}
\newcommand{\M}{\mathcal{Q}}
\newcommand{\bigO}{\mathcal{O}}

\newcommand{\Pm}{P}
\newcommand{\Km}{K}

\newcommand{\E}{\mathbb{E}}
\renewcommand{\Pr}{\mathbb{P}}

\newcommand{\ine}{\mathbb{1}_{\mathcal{E}}}
\newcommand{\notine}{\mathbb{1}_{\domain\backslash \mathcal{E}}}

\newcommand{\eps}{\varepsilon}
\newcommand{\Meff}{M_{\mathrm{eff}}}
\newcommand{\phinm}{\|\phi\|_{0}}
\newcommand{\rhonm}{\|\rho\|_{0}}
\newcommand{\Uepsnm}{\|U_\eps\|_{0}}
\newcommand{\W}[2][\infty]{W^{#2,#1}}
\newcommand{\n}{^{(n)}}
\newcommand{\Rl}{R_\lambda}
\newcommand{\Me}{\M_{\eps,n}}
\newcommand{\eL}{e^{\calL n\eps}}

\newcommand{\half}{\tfrac{1}{2}}
\DeclareMathOperator{\diag}{diag}
\DeclareMathOperator{\arcsinh}{arcsinh}
\DeclareMathOperator*{\Lip}{Lip}

\DeclareMathOperator*{\intr}{int}
\DeclareMathOperator*{\supp}{supp}
\def\R#1{$(\ref{#1})$} 	
\def\RR#1#2{$(\ref{#1}-\ref{#2})$} 	

\newcolumntype{L}{>{$}l<{$}} 
\newcolumntype{R}{>{$}r<{$}} 

\newtheorem{remark}[theorem]{Remark}

\newcommand{\ri}[1]{\textcolor{black}{#1}}
\newcommand{\rii}[1]{\textcolor{black}{#1}}

\begin{document}

\title{Spectral convergence of diffusion maps: improved error bounds and an alternative normalisation\thanks{Submitted to the editors \today. 
}}
\headers{Spectral convergence of diffusion maps}{Caroline L. Wormell and Sebastian Reich}
\author{Caroline L. Wormell\thanks{%
	Laboratoire de Probabilit\'es, Statistique et Mod\'elisation (LPSM),
	Sorbonne Universit\'e, Universit\'e de Paris ({\email{wormell@lpsm.paris}}).
	} \and Sebastian Reich\thanks{%
	Department of Mathematics, University of Potsdam
}}

\maketitle

\begin{abstract}
Diffusion maps is a manifold learning algorithm widely used for dimensionality reduction. Using a sample from a distribution, it approximates the eigenvalues and eigenfunctions of associated Laplace-Beltrami operators. Theoretical bounds on the approximation error are however generally much weaker than the rates that are seen in practice. This paper uses new approaches to improve the error bounds in the model case where the distribution is supported on a hypertorus. For the data sampling (variance) component of the error we make spatially localised compact embedding estimates on certain Hardy spaces; we study the deterministic (bias) component as a perturbation of the Laplace-Beltrami operator's associated PDE, and apply relevant spectral stability results. Using these approaches, we match long-standing pointwise error bounds for both the spectral data and the norm convergence of the operator discretisation.

We also introduce an alternative normalisation for diffusion maps based on Sinkhorn weights. This normalisation approximates a Langevin diffusion on the sample and yields a symmetric operator approximation. We prove that it has better convergence compared with the standard normalisation on flat domains, and present a highly efficient algorithm to compute the Sinkhorn weights.
\end{abstract}

\section{Introduction}\label{s:Intro}

Many problems in data science revolve around the extraction of information about the geometry of some probability distribution given only a sample that may possibly be embedded in an ambient space of much higher dimension: examples of these problems include clustering and dimension reduction. The intrinsic geometry of such a distribution may be encoded by various weighted Laplace-Beltrami operators, from whose spectral data various desiderata can be extracted: for example, the operator's eigenfunctions may be used to define intrinsic coordinates for the support of the distribution \citep{Coifman05,Coifman06}, or may be used in spectral clustering algorithms \citep{Nadler06}.

Diffusion maps is a widely-used algorithm to recover the relevant eigendata \citep{Coifman05,Coifman06}: the idea is to construct a particle discretisation of the evolution of a weighted Laplace-Beltrami operator $\calL$ over some short timestep $\eps$. 
To this end, a kernel matrix $\Km$ is first constructed:
\begin{align} \Km = \left(\tfrac{1}{M} k_\eps(d(x^i,x^j))\right)_{i,j = 1,\ldots, M}, \label{eq:KernelMatrix} \end{align}
where the $x^i \sim \rho\, \dd x$ are the sample points, $k_\eps$ is a symmetric probability kernel with covariance matrix $\eps I$. The kernel matrix is then normalised to be Markov (i.e.~row-stochastic)
\begin{align} \Pm = \diag(\Km u)^{-1} \Km \diag(u),
\label{eq:WeightedMatrix} \end{align}
for some appropriately chosen weight vector $u \in \mathbb{R}^M$.

As the sample size $M$ is taken to infinity and the diffusion timestep $\eps$ is taken to zero with an appropriate dependence on $M$ \citep{Lindenbaum17}, the spectral data of $\Pm$ should approximate that of the Laplace-Beltrami operator semigroup $e^{\eps\calL}$, enabling reconstruction of the spectral data of the operator $\calL$ itself. (Indeed, this problem is often formulated as the graph Laplacian $L = \eps^{-1}(\Pm - I)$ approximating $\calL$.) 
The Markov nature of the normalised matrix $\Pm$ means that the intrinsic coordinates provided by its leading eigenvectors faithfully reconstruct the intrinsic geometry of the distribution's support \citep{Coifman06}.

Standard choices of weights for these operators are of the form $\check u_\alpha \equiv (\Km 1)^{-\alpha}$, for some $\alpha \in [0,1]$. In this case, the weighted Laplace-Beltrami operators to which the convergence occurs are
\begin{equation} \check \calL_\alpha \phi := \half \Delta \phi + (1-\alpha) \nabla \log \rho \cdot \nabla \phi = \half \rho^{-(2-2\alpha)}\, \nabla \cdot (\rho^{2-2\alpha} \nabla \phi), \label{eq:GeneratorStandard} \end{equation}
where $\rho$ is the density of the distribution with respect to Lebesgue measure. The case $\alpha = 0$ (i.e.~$\check u \equiv 1$) is the standard graph Laplacian normalisation; on the other hand, we recover for $\alpha = 1$ the unweighted Laplace-Beltrami operator, and for $\alpha = \half$ the generator of the Langevin diffusion with invariant measure $\rho$ \citep{Coifman06}.
\\

The last twenty years have seen a range of rigorous work establishing and bounding the convergence of diffusion maps and related methods. Because both a space and time discretisation occur, the error decomposes into two parts: a ``variance'' error of finite samples size $M$ with the timestep $\eps$ held fixed, and a ``bias'' error from the positive timestep $\eps$, which is independent of $M$. (Often instead of the timestep, the kernel bandwidth $\eps^{1/2}$ is used, sometimes notated by $h$ or $\epsilon$.) For pointwise estimates on the kernel matrix $\Pm$, the errors associated with the two limits have been shown to be bounded respectively by $\bigO(M^{-1/2}\eps^{-d/4})$ \citep{Hein05} and, on flat manifolds, $\bigO(\eps^2)$ \citep{Singer06}. There are clear intuitions to these error rates: the first is a central limit theorem error between $\Km$ and its infinite data limit, taking into account that of the $M$ sample points, we expect $\Meff = \bigO(M \eps^{d/2})$ to be in the effective support of the kernel; the second is a standard first-order discretisation error for a diffusion operator over timestep $\eps$. It is natural to expect that the pointwise error of the discretisation should transfer to the spectral data: with the short timestep magnifying the errors by a factor of $\bigO(\eps^{-1})$, this would yield an $\bigO(M^{-2/(8+d)})$ error for the optimal scaling of $\eps$ with $M$. 

However, theoretical estimates for spectral data in the literature have been much weaker than this. The standard bound on the bias error in the spectral data, in both $L^2$ and $L^\infty$ norms, has been the naive estimate of $\bigO(\eps^{1/2})$, corresponding to the $L^p \to L^p$ operator error \citep{Hein05,Shi15,Trillos19,Lu20, Dunson19}. 

While the decay of the variance error as $M\to\infty$ with $\eps$ fixed has been long known using compact embedding of Glivenko-Cantelli function classes \citep{vonLuxburg04,vonLuxburg08,Belkin07,Dunson19}, this approach has yielded only weak quantitative bounds on the variance error, the best to date being $\bigO(M^{-1/2} \eps^{-d-3})$ in $L^\infty$ \citep{Shi15}. Due to the dependence of the weights $\check u_\alpha$ on the sample for $\alpha \neq 0$, this approach has also largely been specialised to the graph Laplacian normalisation $\alpha = 0$.
More recently optimal transport techniques have been applied to bound the variance error. These necessarily sacrifice the central limit theorem convergence in $M$ for the much slower optimal transport rate of $\bigO((\tfrac{\log M}{M})^{1/d})$, but yield an overall error of $\bigO((\tfrac{\log M}{M})^{1/2d})$ in the eigenvalues for dimensions $d \geq 2$ \citep{Trillos19, Lu20}. 

In \citet{Calder19} these results were bootstrapped with (weaker) pointwise estimates to obtain a central limit theorem convergence in $M$ with overall $L^2$-convergence rate of $\bigO((\tfrac{\log M}{M})^{1/(d+4)})$ on general manifolds (although this was based on an assumption of $O(\eps^{1/2})$ bias error on curved manifolds). However, only unweighted graph Laplacians were studied: more complex kernel estimation problems will require increasingly more complex concentration of measure estimates to obtain pointwise convergence. Furthermore, the spectral convergence was obtained from pointwise convergence via Rayleigh quotients, which are specific to self-adjoint operators. 

The first goal of our paper is to prove that for diffusion maps normalisations, the pointwise error bounds hold for the spectral data. This work is independent of \citet{Calder19} and takes a different, more dynamical approach that may in fact be applied very generally to Gaussian kernel-based discretisation problems. This is because we fully carry through the pointwise convergence rates of diffusion maps discretisations to norm convergence of the discretised operators. 
For simplicity, we will assume the support of the measure is a flat torus $\domain = (\mathbb{R}/L\mathbb{Z})^d$ and the sample points $x^i$ are independent and identically distributed; we will use the standard Gaussian choice of kernel. Our only assumptions on the sample density $\rho$ are that it is bounded away from zero and $C^{3/2+\beta}$ H\"older for some $\beta > 0$ (i.e.~a $C^{\beta'}$ first derivative for some $\beta > 1/2$). 
We will show that convergence of eigenfunctions holds in the space $C^0$ of continuous functions (i.e. in $L^\infty$ norm).

To achieve this goal we will apply new approaches to both the bias and variance components of the error. To bound the bias error, we will reformulate the problem as one of compact PDE evolution operators for which the perturbations are bounded from a strong norm to a weak norm, and apply the relevant spectral approximation theory \citep{Keller99}. \rii{For the standard weights, this is a more or less straightforward result in approximation of diffusion semigroups, except that we apply the theory of negative Sobolev spaces obtain convergence for $\rho$ of relatively low regularity. For the Sinkhorn weights discussed below, we will combine this with an averaging argument to prove the faster convergence.}

On the other hand, our bounds on the variance error conservatively extend the pointwise error bound for kernels to operator errors in certain Hardy spaces via localised compact embedding estimates. \rii{These embedding estimates are obtained by considering Glivenko-Cantelli function classes on small subsets of the domain and harnessing the localisation of the kernel. The enabling factors in our techniques are thus the kernel function's smoothness (in Proposition \ref{p:Compactness}) and its fast decay away from zero (in Proposition \ref{p:LInfTails}); these results do not rely on differentiability of the sampling measure or of the underlying manifold, nor in fact do they rely on the Markov nature of the operator.
}

Combining these, we will obtain \ri{a spectral error} of $\bigO(M^{-1/2} \eps^{-1-d/4} (\log M \eps^{-1})^{d-1/2} + \eps)$. \ri{With} optimal scaling $\eps \sim M^{-2/(8+d) + o_M(1)}$, this gives a total error of $\bigO(M^{-2/(8+d) + o_M(1)})$. For larger dimensions $d \geq 3$, this is a major improvement over previous results for weighted Laplacians: for example, compared with \citet{Trillos19} the accuracy is squared for $d = 8$. It is also a significant improvement on the unweighted Laplacian results of \citet{Calder19}. Our convergence rate for the variance error of spectral data estimates still remains somewhat weaker than variance errors observed empirically. \rii{This is partly because we obtain convergence results for normalised operators directly from the convergence of the unnormalised kernel}: we expect that making use of the Markov nature of the semigroup will give an $\bigO(\eps^{1/2})$ improvement in variance error \citep{Singer06},  \rii{bringing it into line with previous results in the regime of pointwise convergence \citep{Calder19}, to which our $L^\infty$ results must necessarily be limited}. 
On the other hand, the $\bigO(\eps)$ bias error bound appears optimal. These rates can be expected to carry across to \rii{general} curved manifolds (for the bias error this was observed in \citet{Vaughn19}). 
\\

Our theoretical approach facilitates the second goal of the paper: to 
study a superior normalisation using Sinkhorn weights for the Langevin dynamics whose generator is
\begin{equation} \calL \phi := \check \calL_{0.5}\phi = \half \Delta \phi + \half \nabla \log \rho \cdot \nabla \phi. 
\label{eq:Generator} \end{equation}

So-called Sinkhorn weights $u,\, 1/(\Km u)$ for a general matrix $\Km$ are defined to be those making the row-stochastic matrix $\Pm = \diag(\Km u)^{-1} \Km \diag(u)$ also column-stochastic.
This kind of matrix weighting problem has been studied since \citet{Sinkhorn64}; it has been studied in the context of image processing \citep{Kheradmand14} and in spectral clustering \citep{Brand03,Wang12,Wang16,Wang20}; it has seen interest in the context of computing entropically regularised optimal transport plans \citep{Cuturi13,Altschuler17,Feydy19}, and has been recently been considered as a diffusion maps normalisation \citep{Marshall19}. The Sinkhorn (also known as bi-stochastic or doubly stochastic) normalisation has been found to have superior properties to many standard Markovian kernels in many applications \citep{Wang20,Landa20}. In using Sinkhorn weights as a normalisation for diffusion maps, we will study the restricted case where the kernel matrix $\Km$ is symmetric (and so one is computing a coupling of the sample's empirical measure $\rho^M$ with itself). 
In this restricted case, such weights solve the (quadratic) problem
\begin{align} \diag(u) = \diag(\Km u)^{-1}.
\label{eq:SinkhornWeightsDiscrete} \end{align}

We will prove that, at least in the cases we consider, the Sinkhorn weights have an improved rate of convergence, with the bias error in eigendata improving to $\bigO(\eps^2)$ from $\bigO(\eps)$ for standard weights, and the variance error remaining the same. This means that, compared with the standard weights, a larger timestep $\eps \sim M^{-2/(12+d) + o(1)}$ may be chosen with a further improved overall convergence rate of $\bigO(M^{-4/(12+d) + o(1)})$, although in practice $\eps$ has to be rather small, and thus $M$ very large, for this convergence rate to take hold. For this convergence to hold it is only necessary that the density $\rho$ be $C^{2+\beta}$ H\"older for $\beta > 0$.


While Sinkhorn weights must be computed iteratively, we also present an accelerated algorithm to calculate the weights that, by harnessing the symmetric nature of the problem, converges in $\bigO(1)$ matrix-vector multiplications. This algorithm was first noted by \citet{Marshall19}: here we establish a convergence rate, with rigorous bounds. As a result, use of Sinkhorn weights has minimal numerical overhead.
\\


This paper is structured as follows. In Section \ref{s:Setup} we define the mathematical objects used in the paper; in Section \ref{s:Theorems} we state the main theorems with a brief numerical illustration; in Section \ref{s:Algorithms} we describe our accelerated Sinkhorn algorithm. We then turn to studying the convergence of relevant operators as the timestep $\eps\to 0$, focussing on the more interesting case of the Sinkhorn normalisation. After stating some relevant functional-analytic results in Section \ref{s:FunctionResults} and describing the convergence of Sinkhorn weights as $\eps \to 0$ in Section \ref{s:DeterministicSinkhorn}, we prove the necessary operator convergence result for the bias error in Section \ref{s:DeterministicConvergence}. We then consider the variance error, i.e.~ that of finite $M$: in Section \ref{s:Particle} we bound the operator convergence of the kernel matrix $\Km$ to a continuum limit in appropriate norms, and in Section \ref{s:ParticleSinkhorn} we do the same for the normalised matrix $\Pm$; in Section \ref{s:Spectral} we combine the two operator convergence results to prove the convergence of spectral data for the Sinkhorn weight case. Finally, we outline the corresponding results for standard weights in Section \ref{s:StandardWeights}.

\section{Notation} \label{s:Setup}

We now present some notation that will be used in the main theorems and throughout the paper. \ri{Notation used throughout this paper is tabulated in Tables \ref{t:Table1}-\ref{t:Table2}.}

\begin{table}[h] \label{t:Table1}
	\begin{tabular}{Rl}
		\hline 
		d & The dimension of the domain \\
		L & The side length of the domain \\
		\domain & The domain of the problem, a hypertorus (see Section \ref{ss:Operators}) \\
		\rho & The sampling density \\
		\eps & The timestep parameter; $\eps^{1/2}$ the kernel bandwidth \\
		M & The number of data points in the sample \\
		\Meff & The effective number of points in the bandwidth of the kernel \\
		x^i & A data point sampled from $\rho$ \\
		\rho^M & The empirical measure of the sample $\{x_i\}_{i=1,\ldots, M}$ \\
		\alpha & The parameter for diffusion maps weights (see Section \ref{s:Intro}) \\
		\beta & A H\"older parameter in $(0,1)$ \\
		\delta & The norm of an operator quantifying the variance error (see \R{eq:DeltaDef}) \\
		r,s & Sobolev space differentiability parameters \\
		\lambda_* & A ceiling on the magnitude of eigenvalues considered for convergence \\
		\zeta & The width in the complex direction of $\domain_\zeta$ \\
		Z_0 & The constant of the scaling between $\zeta$ and $\eps^{1/2}$ \\
		\domain_\zeta & A complex fattening of the real domain $\domain$ by $\zeta$ (see Section \ref{ss:FunctionDef}) \\
		g_\eps & The Gaussian kernel on $\mathbb{R}^d$ \\
		g_{\eps,L} & The periodised Gaussian kernel on $\domain$ (see Section \ref{ss:Operators}) \\
		\gamma_{\eps,L}, \gamma'_{\eps,L} & Small constants relating to the periodisation of the Gaussian kernel (see \ref{p:GaussianL}) \\
		\sigma & The square root of the density, $\rho^{1/2}$ (see Section \ref{s:DeterministicSinkhorn}) \\
		u^{(n)} & The $n$th iterate of matrix-vector Sinkhorn iteration (see Section \ref{s:Algorithms}) \\
		U^{(n)} & The $n$th iterate of Sinkhorn iteration in function space (see Sections 	\ref{s:Algorithms} and \ref{s:DeterministicSinkhorn}) \\		
		w^t_\eps & The odd limit cycle of \R{eq:SinkhornPDE} \\
		\hat w^t_\eps & The periodic drift term in the PDE \R{eq:SDMPDE} \\
		\calL & The generator of the Langevin diffusion PDE \R{eq:Generator} \\
		\check \calL_\alpha & The generator of the diffusion PDE \R{eq:GeneratorStandard} \\
		S_\eps & The solution operator of the PDE \R{eq:SDMPDE} \\
		\check S_{\alpha,\eps} & The solution operator of the PDE \R{eq:SDMPDEStandard} \\
		J & The Bessel operator $I - \Delta$ (see Section \ref{ss:FunctionDef}) \\
		K^\cdot_\cdot & Various constants regarding Sobolev space inclusions (see Proposition \ref{p:GeneratorBounds})\\
		Z & The vector space of functions that integrate to zero \\
		H^\infty(\domain_\zeta) & The Hardy space of bounded analytic functions on $\domain_\zeta$ (see Section \ref{ss:FunctionDef}) \\
		R_\lambda(\cdot) & The resolvent at $\lambda$ (see \R{eq:ResolventDef}) \\
		\|\cdot\|_0 & The $C^0$ norm \\
		\|\cdot\|_\zeta & The $H^\infty(\domain_\zeta)$ norm \\
		\|\cdot\|_{s,p} & The $W^{s,p}$ norm \\
		d_{C^0}(\cdot,\cdot) & The $C^0$ distance between eigenspaces (see \R{eq:dCO})\\
		\hline 
	\end{tabular} 
\caption{List of symbols (see also Table \ref{t:Table2})}
\end{table}
\begin{table}[h] \label{t:Table2}
	\begin{tabular}{RRRRRl}
		\hline 
		\multirow{2}{1.0cm}{Gene-rator} &\multirow{2}{1.0cm}{Semi-group}  & \multirow{2}{1.0cm}{Finite $\eps$} & \multirow{2}{1.0cm}{Finite $M, \eps$} & \multirow{2}{1.0cm}{Matx/ vector} \\
		\\
 		\hline
		\Delta/2 & \C_\eps &&&& The Gaussian diffusion operator (see \R{eq:Convolution}) \\
		&& \D_\eps & \D^M_\eps & K & The unweighted kernel operator (see (\ref{eq:DMeps}, \ref{eq:Deps}, \ref{eq:KernelMatrix})) \\
		&& \rho_\eps & \rho^M_\eps && The measure, slightly diffused (see \R{eq:RhoEpsDef}) \\
				\smallskip\\
		\multicolumn{3}{l}{\bf Sinkhorn weights}\\
		&& U_\eps & U^M_\eps & u & The weight function (see Section \ref{ss:Operators}, \R{eq:SinkhornWeightsDiscrete}) \\
		\calL & e^{\eps \cal L} & \P_\eps & \P^M_\eps & P & The (Sinkhorn) weighted operator (see (\ref{eq:SemigroupDiscrete}, \ref{eq:SemigroupDeterministic}, \ref{eq:WeightedMatrix})) \\
		&& Y_\eps & Y^M_\eps && The half-step (Sinkhorn) weight (see \R{eq:Y}) \\
		&& \J_\eps & \J^M_\eps && The left half-step operator (see \R{eq:J}) \\
		&& \K_\eps & \K^M_\eps  && The right half-step operator (see \R{eq:K}) \\
		&& \M_{\eps,n} & \M^M_{\eps,n} && The semiconjugacy of $\P^n$ by half-step (see \R{eq:M}) \\
		-\lambda_k & e^{-\eps \lambda_k} & e^{- \eps \lambda_{k,\eps}} & e^{- \eps \lambda^M_{k,\eps}} && The $k$th eigenvalue of $\calL$/$\P$ (see Section \ref{ss:Eigenspaces}) \\
		- \lambda_k & & -\tilde \lambda_{k,\eps} & -\tilde \lambda^M_{k,\eps} && The $k$th  graph Laplacian eigenvalue (see Section \ref{eq:GraphLaplacianEigenvalues}) \\
		E_k & E_k & \bar{E}_{k,\eps} & \bar{E}^M_{k,\eps} && The $k$th eigenspace of $\calL$/$\P$ (see Section \ref{ss:Eigenspaces}) \\
		\Pi_k & \Pi_k & \Pi_{k,\eps} & \Pi^M_{k,\eps} && The $k$th spectral projection operator (see \RR{eq:SemigroupProjectionDefinition}{eq:ProjectionDefinition}) \\
		\smallskip\\
		\multicolumn{3}{l}{\bf Standard weights}\\
		&& \check U_{\eps,\alpha} & \check U^M_{\eps,\alpha} & \check u_\alpha & The right-hand (standard) weight function (see \R{eq:StandardDiscretisationU}) \\
		&& \check V_{\eps,\alpha} & \check V^M_{\eps,\alpha} & & The left-hand (standard) weight function (see \R{eq:StandardDiscretisationV}) \\
		\check \calL_\alpha & e^{\eps \check\calL_\alpha} & \check\P_{\eps,\alpha} & \check\P^M_{\eps,\alpha} & & The (standard) weighted operator (see \R{eq:StandardDiscretisationP}) \\
		&& \check Y_\eps & \check Y^M_\eps && The half-step weight function (see \R{eq:YStandard}) \\
		&& \check\J_{\eps,\alpha} & \check \J^M_{\eps,\alpha} && The left half-step operator (see \R{eq:JStandard}) \\
		&& \check \K_{\eps,\alpha} & \check \K^M_{\eps,\alpha}  && The right half-step operator (see \R{eq:KStandard}) \\
		&& \check \M_{\eps,n} & \check \M^M_{\eps,n} && The semiconjugacy of $\check\P^n$ by half-step (see \R{eq:MStandard}) \\
		-\check \lambda_{k,\alpha} & e^{-\eps \check \lambda_{k,\alpha}} & e^{- \eps \check \lambda_{k,\eps,\alpha}} & e^{- \eps \check \lambda^M_{k,\eps,\alpha}} && The $k$th eigenvalue of $\check\calL$/$\check\P$ (see Section \ref{ss:Eigenspaces}) \\
		- \check \lambda_{k,\alpha} & & -\check{\tilde \lambda}_{k,\eps,\alpha} & -\check{\tilde \lambda}^M_{k,\eps,\alpha} && The $k$th graph Laplacian eigenvalue (see \R{eq:GraphLaplacianEigenvaluesStandard}) \\
		\check E_{k,\alpha} & \check E_{k,\alpha} & \check{\bar{E}}_{k,\eps,\alpha} & \check{\bar{E}}^M_{k,\eps,\alpha} && The $k$th eigenspace of $\check\calL$/$\check\P$ (see Section \ref{ss:Eigenspaces}) \\
		\hline 
	\end{tabular} 
\caption{List of symbols dependent on different limits. Where a reference is not given for standard weights quantities, they are defined by analogy with the Sinkhorn equivalent.}
\end{table}

\subsection{Operators}\label{ss:Operators}

\ri{Recall that our domain is $\domain = (\mathbb{R} / L\mathbb{Z})^d$.} 
We will use as our kernel function the periodic Gaussian kernel $k_\eps(x,y) = g_{\eps,L}(y-x)$:
\begin{align} g_{\eps,L}(x) &= \sum_{\mathbf{j} \in \mathbb{Z}^d} g_\eps(x + L\mathbf{j}), \label{eq:GaussianKernelL}\end{align}
where the standard Gaussian kernel is
\begin{align} g_\eps(x) &= (2\pi \eps)^{-d/2} e^{-\|x\|^2/2\eps}. \label{eq:GaussianKernel}\end{align}
Note that if, as is typical, the bandwidth $\sqrt{\eps} \ll L$, all but one summand in \R{eq:GaussianKernelL} will be superexponentially small.

We define convolution by the periodic Gaussian kernel \R{eq:GaussianKernelL} as an operator
\begin{equation} (\C_\eps\phi)(x) = \int_\domain g_{\eps,L}(y-x) \phi(y)\, \dd x, \label{eq:Convolution}\end{equation}
which has the semigroup property $\C_s \C_t = \C_{s+t}$.

In this paper we will interpolate the vectors and matrices introduced in the introduction by functions defined on the continuous domain $\domain$. Our interpolation arises very naturally: the kernel matrix $\Km$ defined in \R{eq:KernelMatrix} acting on vectors $(\phi(x^i))_{i = 1,\ldots,M}$ extends to the following operator
\begin{equation} (\D^M_\eps \phi)(x) := \frac{1}{M} \sum_{i=1}^M g_{\eps,L}(x-x^i) \phi(x^i) = (\C_\eps \rho^M \phi)(x), \label{eq:DMeps} \end{equation}
where $\rho^M$ is the empirical measure of the sample. In particular, if we define the restriction to sample points $\varpi(\phi) = (\phi(x_i))_{i = 1,\ldots, M}$, then $\varpi \circ \D^M_\eps = \Km \circ \varpi$.

For Sinkhorn weights our weight vector $u$, defined in \R{eq:SinkhornWeightsDiscrete}, then extends to the function $U^M_\eps$ given as the unique solution of
\begin{equation} U^M_\eps(x)\, (\D^M_\eps U^M_\eps)(x) \equiv 1, \label{eq:SinkhornDiscretisation} \end{equation}
so $\varpi(U^M_\eps) = u$.

Our normalised matrix then extends to the operator
\begin{align}
(\P^M_\eps \phi)(x) &= U^M_\eps(x) (\D^M_{\eps}U^M_\eps\phi)(x). \label{eq:SemigroupDiscrete}
\end{align}
Since $P \circ \varpi = \varpi \circ \P^M_\eps$, $P$ and $\P^M_\eps$ will have identical (non-zero) spectra and identical eigenvectors (up to $\varpi$).

We will study a range of weighted operators of a form similar to \R{eq:SemigroupDiscrete} and we will write them for short in the following manner:
\begin{align*}
\P^M_\eps &= U^M_\eps \D^M_{\eps}U^M_\eps.
\end{align*}

In this paper we are required to consider two limits and their associated errors: the stochastic, so-called ``variance'' error as the finite sample size $M\to\infty$ for fixed timestep $\eps$, and the deterministic, so-called ``bias'' error, in the spatial continuum limit as the timestep $\eps \to 0$. We will show in Section \ref{s:Particle} that the discrete kernel operator $\D^M_\eps$ converges in the $M\to\infty$ data limit to a continuum kernel operator
\begin{equation} (\D_\eps\phi)(x) = \int g_\eps(x-y) \phi(y) \rho(y) \,\dd y = (\C_\eps\rho\phi)(x) \label{eq:Deps}. \end{equation}

In the infinite data limit we will show in Section \ref{s:ParticleSinkhorn} that the $U^M_\eps$ converge to functions $U_\eps$ that satisfy a continuum version of the Sinkhorn problem
\begin{equation} U_\eps(x)\, (\D_\eps U_\eps)(x) \equiv 1. \label{eq:SinkhornProblem} \end{equation}

From this we have a deterministic approximation to the semigroup $e^{\eps \calL}$
\begin{align}
\P_\eps &= U_\eps \D_{\eps} U_\eps, \label{eq:SemigroupDeterministic}
\end{align}
to which we expect the normalised discrete operator $\P^M_\eps$ to converge.

Because the two limits require the use of different function spaces to attain the appropriate convergence rates we will consider semi-conjugacies of our operators $\P^M_\eps$ and $\P_\eps$ that will be bounded on the space of continuous functions $C^0$. For concision, in this discussion we will take ``$\mathcal{A}^{(M)}_\eps$'' to mean ``$\mathcal{A}_\eps$ (resp. $\mathcal{A}^M_\eps$)''.

Using that $\D^{(M)}_{\eps} = \C_{\eps/2} \D^{(M)}_{\eps/2}$ we will define the half-step weight functions
\begin{align}
Y^{(M)}_\eps(x) = (\D^{(M)}_{\eps/2} U^{(M)}_\eps)(x) \label{eq:Y}
\end{align}
and the half-step operators
\begin{align}
\J^{(M)}_\eps &= U^{(M)}_\eps \C_{\eps/2} Y^{(M)}_\eps \label{eq:J} \\
\K^{(M)}_\eps &= (Y^{(M)}_\eps)^{-1} \D^{(M)}_{\eps/2} U^{(M)}_\eps. \label{eq:K}
\end{align}
These operators $\J^{(M)}_\eps, \K^{(M)}_\eps$ are positive, preserve constant functions and have
\[ \P^{(M)}_\eps = \J^{(M)}_\eps \K^{(M)}_\eps. \]
We then define the following operators that are semi-conjugate to $(\P^M_\eps)^n$
\begin{align}
\M^{(M)}_{n,\eps} &= (\K^{(M)}_\eps \J^{(M)}_\eps)^n \label{eq:M}.
\end{align}
\\

To study the situation for the standard weights, we will define the kernel density estimate of the distribution using the sample:
\begin{equation} \rho^{(M)}_\eps = \D^{(M)}_\eps 1 \label{eq:RhoEpsDef} \end{equation}
The weight vectors $\check u_\alpha := (\Km 1)^{-\alpha}$ and $1/(\Km \check u_\alpha)$ then extend respectively to the functions
\begin{align}
\check{U}^{(M)}_{\eps,\alpha}(x) &= (\rho^{(M)}_\eps(x))^{-\alpha}. \label{eq:StandardDiscretisationU} \\
\check{V}^{(M)}_{\eps,\alpha}(x) &= 1/(\D^{(M)}_\eps\check{U}^{(M)}_{\eps,\alpha})(x). \label{eq:StandardDiscretisationV}
\end{align}
We then have the approximations to the semigroup $e^{\eps \check{\calL}_\alpha}$
\begin{equation} \check{\P}^{(M)}_{\eps,\alpha} = \check{V}^{(M)}_{\eps,\alpha} \D^{(M)}_{\eps,\alpha} \check{U}^{(M)}_{\eps,\alpha}, \label{eq:StandardDiscretisationP}
\end{equation}
We then define operators $\check{\J}^{(M)}_{\eps,\alpha}, \check{\K}^{(M)}_{\eps,\alpha}, \check{\M}^{(M)}_{\eps,\alpha}$ analogously to the Sinkhorn weight case (see \RR{eq:YStandard}{eq:MStandard}).

\subsection{Function spaces}\label{ss:FunctionDef}

We will use two different classes of function spaces to study the bias and variance error. To study the variance error, we will need spaces with very strongly compact embeddings into $C^0$, specifically Hardy spaces of analytic functions. On the other hand, when considering the bias error we are comparing against the semigroup $e^{\eps \calL}$, and because of our relaxed conditions on the regularity of $\rho$, we can only expect the image of the semigroup to be contained in spaces of low differentiability.

To study the bias error, we will therefore make use of the scales of Sobolev spaces $\W[p]{s} \subseteq L^p(\domain,\dd x)$ for $s \geq 0, p \in (1,\infty]$, which each consist of function classes $[\phi]$ for which the norm
\[ \| \phi \|_{\W[p]{s}} := \| J^{s/2} \phi \|_{L^p},\]
is finite and well-defined, where the operator $J = I - \Delta$. For some operator $\mathcal{A}$ that is sectorial (see Section \ref{s:FunctionResults}) and thus for which a semigroup $e^{-\mathcal{A}t}, t\geq 0$ is defined, we define fractional powers as inverses of the injections \citep{Henry06}
\begin{equation} \mathcal{A}^{-s/2} := \Gamma(s/2)^{-1} \int_0^\infty t^{-s/2-1} e^{-t\mathcal{A}}\,\dd t,\, s>0. 
\label{eq:FractionalPower} \end{equation}
The operator $J$ is sectorial on all $\W[p]{s}$ \citep{Haase06}.

For integer $k\geq 0$ the space of $k$-times continuously differentiable functions $C^k$ is a subset of $\W{k}$ with equivalent norms. Furthermore, for all $s' < s$, the H\"older space $C^{s'} \subseteq \W{s}$, and each $[\phi] \in \W{s}$ has an element $\phi \in C^{s}$: the inclusion maps between these function spaces are continuous.
\\

On the other hand, to study the convergence of the particle discretisation (i.e.~the variance error), we will use spaces of bounded analytic functions on narrow strips around the domain $\domain$. We therefore define for $\zeta \geq 0$ the complex domains
\[ \domain_\zeta = \{ x + iz \mid x \in \domain, z \in [-\zeta,\zeta]^d \}, \]
and the corresponding Hardy space
\[ H^\infty(\domain_\zeta) = \{ \phi \in C^0(\domain_\zeta) : \phi \textrm{ analytic on }\intr \domain_\zeta\} \]
with norm
\begin{equation} \| \phi \|_{\zeta} = \sup_{z \in \domain_\zeta} | \phi(z)|. \label{eq:HardyNormDefinition}\end{equation}
Note that the Hardy space norm $\| \cdot\|_{\zeta}$ is always equal to or greater than $\|\cdot\|_0$, the $C^0$ norm on the real domain $\domain$.

These Hardy spaces encode the smoothness of the Gaussian kernel: \rii{unlike if we used a $C^k$ space, this allows for very good local compact embedding results into $C^0$. On the other hand, our choice of thin strips $\domain_\zeta$ as our complex domain allows the kernel to have $\bigO(1)$ norm as an operator $C^0 \to H^\infty(\domain_\zeta)$, if we take $\zeta$ to scale with the $\bigO(\eps^{1/2})$ kernel bandwidth}. 
\\

In this paper we will assume that our measure density $\rho$ is strictly bounded away from zero, and that it lies in the Sobolev space $\W{s}$, where $s > 3/2$ for the standard normalisation and $s>2$ for the Sinkhorn normalisation: it is equivalent to assume that $\rho \in C^{3/2+\beta}$ (resp. $\rho \in C^{2+\beta}$) for some $\beta > 0$.

\subsection{Eigendata}\label{ss:Eigenspaces}

The generator $\calL$ has eigenvalues $0 = -\lambda_0 > -\lambda_1 \geq -\lambda_2 \geq \cdots$, and the semigroup approximations $\P^{(M)}_\eps$ have respective eigenvalues $1 = e^{-\eps\lambda_{0,\eps}^{(M)}} > e^{-\eps\lambda_{1,\eps}^{(M)}} \geq e^{-\eps\lambda_{2,\eps}^{(M)}} \geq \cdots \geq 0$. That is, $-\lambda_{0,\eps}^{(M)},-\lambda_{1,\eps}^{(M)},\ldots$ are the estimates of the Laplacian eigenvalues obtained from the semigroup approximations.

Note that the non-negativity of these eigenvalues is guaranteed via positive semi-definiteness of $\P^{(M)}_\eps$ in $L^2(\rho^{(M)})$. We denote the corresponding eigenspaces $E_k, E_{k,\eps}^{(M)}$, and merge discretised eigenspaces whose eigenvalues will converge in the limit:
\[ \bar E_{k,\eps}^{(M)} := \bigoplus_{\lambda_j = \lambda_k} E_{k,\eps}^{(M)}. \]

For the standard weights we define equivalent quantities: \ri{$\check \lambda_{k,\alpha}$ the eigenvalues of $\check \calL_{\alpha}$ with eigenspaces $\check E_{k,\alpha}$, and $e^{-\eps \check\lambda^{(M)}_{k,\eps,\alpha}}$ the eigenvalues of $\check\P^{(M)}_{\eps,\alpha}$ with eigenspaces $\check{\bar{E}}^{(M)}_{k,\eps,\alpha}$ (merged appropriately for degenerate eigenvalues of the limiting generator $\check{\calL}_\alpha$).}
Note that we have positive semi-definiteness of $\check\P^{(M)}_{\eps,\alpha}$ in $L^2(\rho^{(M)}(\check U^{(M)}_{\eps,\alpha}/\check V^{(M)}_{\eps,\alpha})^{1/2})$).

Finally, to quantify the convergence of eigenspaces, we define the distance between vector subspaces:
\begin{equation} d_{C^0}(E,F) = \max\left\{\sup_{\phi \in B_{C^0}(1) \cap E} \inf_{\psi \in F} d_{C^0}(\phi,\psi), \sup_{\phi \in B_{C^0}(1) \cap F} \inf_{\psi \in E} d_{C^0}(\psi,\phi)\right\}. \label{eq:dCO} \end{equation}
This distance $d_{C^0}$ therefore quantifies the distance between subspaces in the $L^\infty$ norm.

	\subsection{Dependency of constants}

All constants in our paper depend only on: the dimension of the manifold $d$, the side-length of the domain $L$, the Sobolev differentiability parameter of the density $s$, the Sobolev norm of the log-density $\| \log \rho \|_{\W{s}}$, and the upper limit on the timestep $\eps_0$.

Constants specifically pertaining to eigendata (i.e. those in Section \ref{s:Spectral}, Theorems \ref{t:EigendataStandard}-\ref{t:Eigendata}, and Corollary \ref{c:GraphLaplacian}) also depend on the eigendata of the generator $\mathcal{L}$ or $\check{\mathcal{L}}_\alpha$, specifically the lower bound  $-\lambda_*$ on the eigenvalues under consideration, and the minimum separation between points in $(\sigma(\mathcal{L}) \cap (-\lambda_*,0]) \cup \{-\lambda_*\} $ (respectively for $\check{\mathcal{L}}_\alpha$).

\section{Main results}\label{s:Theorems}

In this paper we will deterministically bound the ``variance'' errors, which depend on the empirical measure $\rho^M$, exclusively via an operator error $\delta$:
	\begin{equation}
\delta := \left\| \D^M_{\eps/2} - \D_{\eps/2}\right\|_{H^\infty(\domain_\zeta) \to C^0}, \label{eq:DeltaDef}
\end{equation}
where $\zeta = Z_0 \eps^{1/2}$ for some constant $Z_0$, and $H^\infty(\domain_\zeta)$ is defined in \R{eq:HardyNormDefinition}. Thus, results in terms of $\delta$ can be applied to any point sample, including weighted, dependent and deterministic samples.

When the empirical measure is an {\em i.i.d.} sample from the true density $\rho$, we have the following probabilistic bound on $\delta$:
\begin{theorem}\label{t:Delta}
	Suppose $\rho \in L^\infty$. There exist constants ${{C}}_1, {{C}}_{2}$ depending only on $L$, $d$, $\rhonm$, $\eps_0$, $Z_0$ such that for all $\eps < \eps_0$ and $c < \tfrac{3}{4} \rhonm$,
	\begin{equation*} \Pr\left(\delta > c\right) \leq \exp\left\{{{C}}_1 (\log c + \log\eps^{-1})^{2d+1} - {{C}}_{2} M  \eps^{d/2} c^2 \right\}. \label{eq:DeltaBound} \end{equation*}
\end{theorem}

In other words, with very high probability
\[ \delta = \bigO\left(M^{-1/2} \eps^{-d/4} (\log M + \log \eps^{-1})^{d-1/2}\right).\]
\\

We can now state the main theorems, on convergence of spectral data for the diffusion maps approximations:
\begin{theorem}[Spectral convergence for standard weights]\label{t:EigendataStandard}
	Suppose $\rho \in \W{s},\, s > 3/2$. 
	For all $\alpha \in [0,1]$ and $\lambda_* > 0$ there exist constants $\check {{C}}_{3}, \check {{C}}_{4}, \check {{C}}_{5}$ such that if $\eps+  \eps^{-1} \delta < \check {{C}}_{3}$, then for $-\lambda_{k,\alpha} \geq -\lambda_*$ we have
	\begin{enumerate}[(a)]
		\item Convergence of eigenvalues of $\check\P_{\eps,\alpha}$ and $\check\P^M_{\eps,\alpha}$:
		\begin{align*} | \check\lambda_{k,\eps,\alpha} - \check\lambda_{k,\alpha} | &\leq \check {{C}}_{4} \eps \\
		| \check\lambda_{k,\eps,\alpha}^M - \check\lambda_{k,\alpha}| &\leq \check {{C}}_{4} (\eps + \eps^{-1} \delta).
		\end{align*}
		\item Sup-norm convergence of the respective eigenspaces:
		\begin{align*} d_{C^0}(\check{\bar E}_{k,\eps,\alpha}, \check E_{k,\alpha}) &\leq \check {{C}}_{5} \eps \\
		d_{C^0}(\check{\bar E}_{k,\eps,\alpha}^M, \check E_{k,\alpha}) &\leq \check {{C}}_{5} (\eps + \eps^{-1} \delta).
		\end{align*}
	\end{enumerate}
\end{theorem}
(Recall that checked quantities $\check \lambda_{k,\alpha}$, etc. are defined analogously to their unchecked weights $\lambda_k$, etc. using standard rather than Sinkhorn normalisations.)
\begin{theorem}[Spectral convergence for Sinkhorn weights]\label{t:Eigendata}
	Suppose $\rho \in \W{s},\, s > 2$. 
	For all $\lambda_* > 0$ there exist constants ${{C}}_{3}, {{C}}_{4}, {{C}}_{5}$ such that if $\eps^2 +  \eps^{-1} \delta < {{C}}_{3}$, then for $-\lambda_k \geq -\lambda_*$ we have
	\begin{enumerate}[(a)]
		\item Convergence of eigenvalues of $\P_{\eps}$ and $\P^M_{\eps}$:
		\begin{align*}
		| \lambda_{k,\eps} - \lambda_k | &\leq {{C}}_{4} \eps^2 \\
		| \lambda_{k,\eps}^M - \lambda_k| &\leq {{C}}_{4} (\eps^2 + \eps^{-1} \delta).
		\end{align*}
		\item Sup-norm convergence of the respective eigenspaces:
		\begin{align*}
		d_{C^0}(\bar E_{k,\eps}, E_k) &\leq {{C}}_{5} \eps^2 \\
		d_{C^0}(\bar E_{k,\eps}^M, E_k) &\leq {{C}}_{5} (\eps^2 + \eps^{-1} \delta).
 		\end{align*}
	\end{enumerate}
\end{theorem}

An empirical comparison of the bias errors for the standard and Sinkhorn normalisations on a $C^{2+\beta}$ sampling distribution is given in Figure \ref{fig:biaserror}, demonstrating the optimality of the bias error bounds, and the better convergence of the Sinkhorn normalisation for $\alpha = \half$.

The empirical behaviour of variance errors for a three-dimensional example is given in Figure \ref{fig:varianceerror}. Here the variance error in the spectral data appears to have the central limit theorem convergence in the sample size $M$ that we have shown. However, this convergence occurs in the regime $\Meff = M \eps^{d/2} \gg 1$, i.e.~up to log terms that $\delta \ll 1$: this regime is larger than that covered by our results, $\eps^{-1} \delta \ll 1$. Furthermore, the dependence on the timestep $\eps$ appears to be more gentle than our results would suggest: as $\eps$ is decreased with $\Meff$ fixed, the variance error in fact appears to decrease rather than increasing as $\bigO(\eps^{-1})$. This is in accordance with previous observations that spectral estimates have better convergence than the pointwise estimates that our results match up to \citep{Trillos19, Calder19}.
\\

\begin{figure}
	\centering
	\includegraphics[width=0.6\linewidth]{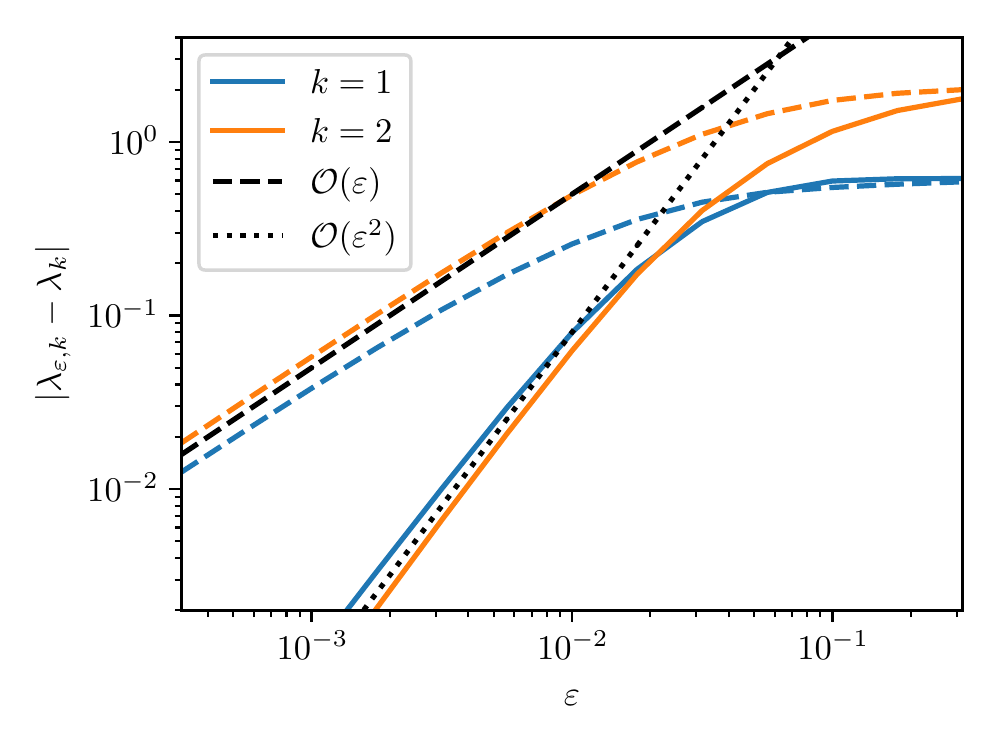}
	\caption{Bias error in eigenvalues for $C^{2.2^-}$ function density $\rho(x) = 1 + \tfrac{1-3^{-2.2}}{2} \sum_{j=1}^\infty 3^{-2.2j} \cos (3^j\cdot 2\pi x)$ on $\domain = \mathbb{R}/\mathbb{Z}$ using a Sinkhorn normalisation (solid lines) and $\alpha = \half$ standard normalisation (dashed lines). A Fourier operator discretisation with $2001$ modes was used to compute the spectrum of the generator $\calL$ and discrete-time approximations $\P_\eps, \check \P_{\eps,1/2}$.}
	\label{fig:biaserror}
\end{figure}

\begin{figure}
	\centering
	\includegraphics[width=0.6\linewidth]{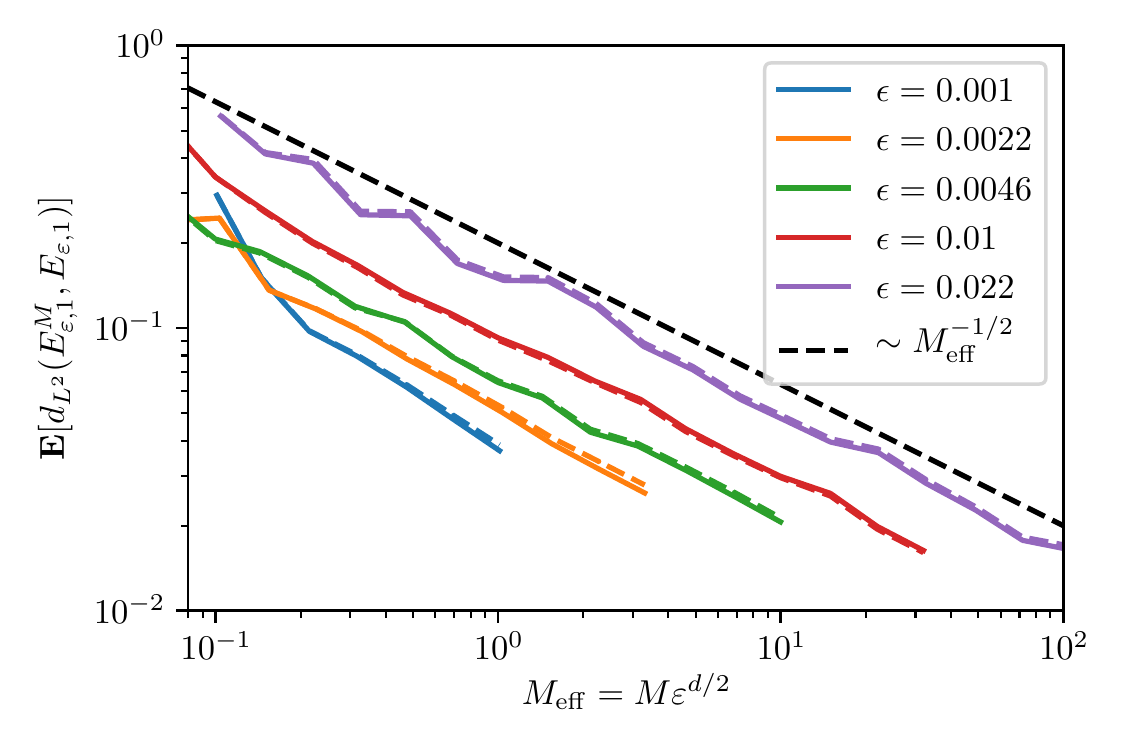}
		\includegraphics[width=0.6\linewidth]{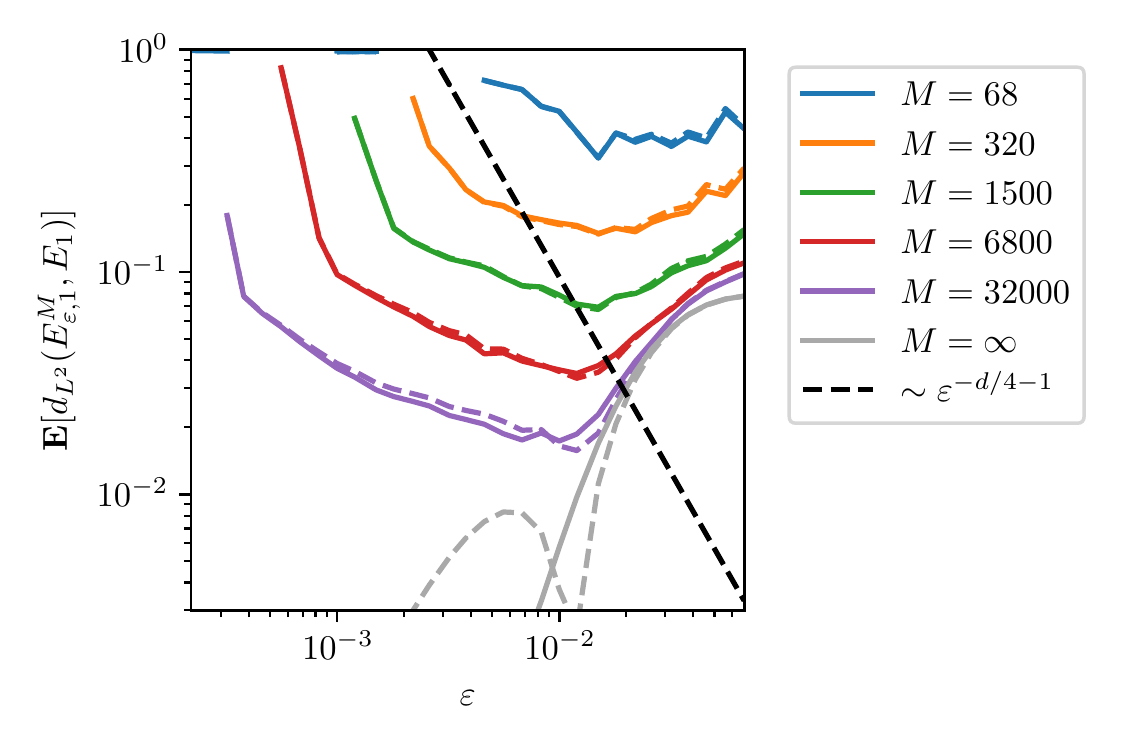}
	\caption[Error in variance estimates]{$L^2(\rho^M)$ error in diffusion maps estimates of eigenspace $E_1$ for function density $\rho(x,y,z) \propto e^{\cos 4\pi x + f(y)+f(z)}$ where $f(x) = 0.4 \cos 2 \pi x+ 0.12 \sin 4 \pi x$ on $\domain = (\mathbb{R}/\mathbb{Z})^3$.

	Sinkhorn normalisation (solid lines) and $\alpha = \half$ standard normalisation (dashed lines) are compared. At top, the variance error plotted against local sample size $\Meff$ for different $\eps$; at bottom the combined bias and variance error are plotted against timestep $\eps$ for different sample sizes $M$.

	Expectations were computed using $30$ samples each. An adaptive Fourier discretisation \citep{ApproxFun} was used to approximate the eigenfunctions of the generator $\calL$ and of the continuum semigroup approximations $\P_\eps, \check \P_{\eps,1/2}$.}
	\label{fig:varianceerror}
\end{figure}

Rather than the semigroup, one is often interested in approximating the Laplace-Beltrami operator $\calL$ itself via the (possibly weighted) graph Laplacian $\eps^{-1}(\Pm -I)$: as an operator, we are thus interested in
\[ \calL^{(M)}_\eps := \eps^{-1}(\P^{(M)}_\eps-I),\]
and similarly for the standard weights. The eigenfunctions of these operators are the same as that of the respective $\P^{(M)}_\eps$, thus with the same convergence. On the other hand, if we let the eigenvalues of the generator $\calL^{(M)}_\eps$ be
\begin{equation} -\tilde \lambda_{k,\eps}^{(M)} = \eps^{-1}(e^{-\lambda_{k,\eps}^{(M)}}-1), \label{eq:GraphLaplacianEigenvalues}\end{equation}
and similarly for standard weights 
\begin{equation} -\check{\tilde \lambda}_{k,\alpha,\eps}^{(M)} = \eps^{-1}(e^{-\check{\lambda}_{k,\alpha,\eps}^{(M)}}-1), \label{eq:GraphLaplacianEigenvaluesStandard}\end{equation}
 then we also have convergence of eigenvalues.
\begin{corollary}[Eigendata of the graph Laplacian]\label{c:GraphLaplacian}
	For all $\alpha \in [0,1]$ and $\lambda_* > 0$ there exist constants ${{C}}_{3}, {{C}}_{6}, \check {{C}}_{3}, \check {{C}}_{6}$ such that
	\begin{enumerate}[(a)]
		\item The eigenspaces of the graph Laplacians $\calL^{(M)}_\eps$ are those of the respective semigroup approximations $\P^{(M)}_\eps$ given in Theorems \ref{t:EigendataStandard} and \ref{t:Eigendata}.
		\item If $\rho \in C^{3/2+\beta}$ and $\eps +  \eps^{-1} \delta < \check {{C}}_{3}$, then for $-\check \lambda_{k,\alpha} \geq -\lambda_*$ we have convergence of eigenvalues for the standard weights
			 \begin{align*} | \tilde{\check \lambda}_{k,\eps,\alpha} - \check \lambda_{k,\alpha} | &\leq \check {{C}}_{6} \eps, \\
			 | \tilde{\check \lambda}_{k,\eps,\alpha}^{M} - \check  \lambda_{k,\alpha}| &\leq \check {{C}}_{6} (\eps + \eps^{-1} \delta).
			 \end{align*}
 		\item  If $\rho \in C^{2+\beta}$ and $\eps^2 +  \eps^{-1} \delta < {{C}}_{3}$, then for $-\lambda_k \geq -\lambda_*$ we have convergence of eigenvalues for the Sinkhorn weights
			 \begin{align*} | \tilde \lambda_{k,\eps} - \lambda_k | &\leq {{C}}_{6} \eps, \\
			 | \tilde \lambda_{k,\eps}^{M} - \lambda_k| &\leq {{C}}_{6} (\eps + \eps^{-1} \delta).
			 \end{align*}
	 \end{enumerate}
\end{corollary}
Note however that for purely linear-algebraic reasons the improvement in the bias error to $\bigO(\eps^2)$ for Sinkhorn weights is lost.
\\

Our proof of the main results rely on bounds of the deviations of (powers of) our discretised half-step operators $\J^{(M)}_\eps, \K^{(M)}_\eps$ from their respective limits.

Thus for the Sinkhorn weights, the bias error is bounded according to the following theorem:
\begin{theorem}\label{t:DeterministicOperator}
	Suppose $\rho \in \W{s},\, s > 2$, and let $S_\eps(t_1,t_0)$ be the solution operator of the PDE
	\begin{equation} \partial_t \phi^t = \calL \phi^t + \nabla \hat w^t_\eps \cdot \nabla \phi^t, \label{eq:SDMPDE}\end{equation}
	where we define $\hat w^t_\eps := \log (\D_t U_\eps) - \half \log \rho$ for $t \in [0,\eps)$ and extend $\eps$-periodically.

	Then
	\begin{align*}
	 	\J_\eps &= S_\eps(\eps,\half\eps)\\
		\K_\eps &= S_\eps(\half\eps,0)\\
		\P_\eps &= S_\eps(\eps,0)\\
		\M_{\eps,n} &= S_\eps((n+\half)\eps,\half\eps).
	\end{align*}
	Furthermore, for all $T>0$ and $\beta \in (0,\min\{s-2,1\})$ there exists a constant ${{C}}_{7,T,\beta}$ such that for all $0 \leq t_1-t_0 \leq T$ and $\eps \leq \eps_0$,
	\begin{align} \|S_\eps(t_1,t_0) - e^{(t_1-t_0) \calL}\|_{C^{3+\beta}\to C^0} \leq {{C}}_{7,T,\beta} \eps^2. \label{eq:GeneratorBound} \end{align}
\end{theorem}

If the sampling density $\rho$ has higher regularity, we have the stronger result, which follows from a simplification of the proof of Theorem \ref{t:DeterministicOperator} and implies an $\bigO(\eps^2)$ pointwise bias error of the Sinkhorn-weighted graph Laplacian:
\begin{proposition}\label{p:DeterministicOperatorS4}
	Suppose $\rho \in \W{s}$ for $s >4$. Then for all $\beta \in (0,1)$ there exists a constant ${{C}}_{8,\beta}$ such that for all $t \in \mathbb{R}$, $\eps \leq \eps_0$,
	\[\| S_\eps(t+\eps,t) - e^{\eps \calL}\|_{C^{3+\beta} \to C^0} \leq {{C}}_{8,\beta}\eps^3. \]
\end{proposition}

This is the best possible asymptotic rate of convergence to the semigroup for operators of the form $V_\eps \D_{\eps} U_\eps$ for all non-uniform distributions $\rho$ (see Remark \ref{r:SinkhornIsBest}).



Bounds on the variance error proceed from Theorem \ref{t:Delta}. In particular, we have the following result on the convergence of the operator $\D^M_\eps$ (an interpolation of the kernel matrix $\Km$) to its continuum limit:
\begin{theorem}\label{t:Operator}
Let $\zeta = Z_0 \eps^{1/2}$. Then
\[ \left\| \D^M_\eps - \D_\eps \right\|_{H^\infty(\domain_\zeta)} \leq e^{2 d Z_0^2} \delta. \]
\end{theorem}
Note here that the imaginary-direction thickness $\zeta$ of the domain of the Hardy space $H^\infty(\domain_\zeta)$ scales proportionally with the $\bigO(\eps^{1/2})$ bandwidth of the kernel. A useful consequence of this is that it is also possible to bound the error of the $k$th derivative of the spatial discretisation, with a penalty in the error of $\bigO(\eps^{-k/2})$.

As a consequence of Theorem \ref{t:Operator}, we also have operator convergence of the normalised operator $\P^M_\eps$, which interpolates the matrix $\Pm$, as well as the various auxiliary operators:
\begin{theorem}\label{t:WeightedOperatorConvergence}
	There exist constants $Z_0, {{C}}_{9}, {{C}}_{10}$ such that if $\zeta = Z_0 \eps^{1/2}$ and $\delta \leq {{C}}_{9}$ 
	then for all $\eps \leq \eps_0$ and $n \in \mathbb{N}$,
	\[ \| \P^M_\eps - \P_\eps \|_\zeta,\, \| \J^M_\eps - \J_\eps \|_{0\to\zeta},\, \| \K^M_\eps - \K_\eps \|_{\zeta\to 0} \leq {{C}}_{10}\delta, \]
	and
	\[ \| \M^M_{\eps,n} - \M_{\eps,n} \|_{0} \leq {{C}}_{10} \delta n. \]
\end{theorem}

\section{Numerical computation of Sinkhorn weights}\label{s:Algorithms}

While the use of Sinkhorn weights gives improved convergence in spectral data, it is necessary to calculate them iteratively: the usual Sinkhorn iteration is known to converge quite slowly in other problems, and indeed substantial efforts have been dedicated to finding ways to accelerate the convergence \citep{Thibault17,Altschuler17,Feydy19,Peyre19}.

However, in our case the extra numerical work necessary to obtain the Sinkhorn weights is small, as in this section we will present a simple, general, well-conditioned algorithm to estimate the Sinkhorn weights that converges exponentially at a rate that is independent of the matrix input.

Let us first note that the traditional way that Sinkhorn weights are calculated is using so-called Sinkhorn iteration: for symmetric matrices this amounts to repeatedly iterating
\[ u^{(n+1)} = 1/(\Km  u^{(n)}),\]
which is interpolated as
\begin{align} U^{(n+1)} = 1/\D^M_\eps [U^{(n)}]. \label{eq:SinkhornIteration} \end{align}
As $n\to\infty$, it is well-known that $U^{(n)} \to c^{(-1)^n} U^M_\eps$ for some constant $c>0$ \citep{Peyre19}. The asymptotic rate of convergence can be bounded, since at the fixed point Sinkhorn iteration is a contraction by $\lambda^M_{\eps,1}$, the second eigenvalue of the re-weighted operator $\P^M_\eps$. This is because the Jacobian at the fixed point is conjugate to $-\P^M_\eps$. However, from Theorem \ref{t:Eigendata}, the spectral gap $1 - \lambda^M_{\eps,1} = \bigO(\eps)$, so $\bigO(\eps^{-1})$ iterates are needed to estimate the Sinkhorn weights to a given tolerance. 

\begin{algorithm}
	\begin{algorithmic}
	\REQUIRE{Unweighted kernel matrix $\Km$, timestep $\eps$, eigendata error tolerance $\tau$}
	\ENSURE{Estimated Sinkhorn weight vector $u$ with log-$L^\infty$ error less than $\eps \tau$}
	\STATE{$u \gets 1/\sqrt{\Km   \mathbf{1}}$}
	\REPEAT
		\STATE{$u_o \gets u$}
		\STATE{$v \gets 1/(\Km  u_o)$}
		\STATE{$u \gets \sqrt{v/(\Km  v)}$}
	\UNTIL{$\|\log (u_o / u)\|_{\ell_2} \leq \eps \tau$}
\end{algorithmic}
	\caption{Accelerated symmetric Sinkhorn algorithm (ASSA)} \label{alg:assa}
\end{algorithm}

To improve this, we propose an accelerated symmetric Sinkhorn algorithm (ASSA, Algorithm \ref{alg:assa}), which harnesses the symmetry and positive definiteness of the iteration problem to accelerate the local convergence rate to $\bigO(8^{-n})$, as well as automatically removing the constant $c$. An iteration step of ASSA involves taking two successive Sinkhorn iterates (c.f. \R{eq:SinkhornIteration}), followed by a geometric mean of the two steps. This algorithm was first noted as a heuristic by \citet{Marshall19}.

We can write this in the case of a kernel operator $\D$ as
\begin{align}
U^{(n)}_a &= 1/\D[U^{(n)}] \label{eq:ASSA1}\\
U^{(n)}_b &= 1/\D[U^{(n)}_a] \label{eq:ASSA2}\\
U^{(n+1)} &= \sqrt{U^{(n)}_a U^{(n)}_b}. \label{eq:ASSA3}
\end{align}

Because the Jacobian of a Sinkhorn iteration step \R{eq:SinkhornIteration} around the fixed point $U$ is conjugate to $-\P := - U \D U$, the Jacobian of the ASSA step is conjugate to $-\half \P (I-\P)$. In our case $\P = \P^M_\eps$ is a self-adjoint, positive definite Markov operator on $L^2(\rho^M)$, so its spectrum is contained in $[0,1]$ and so the spectrum of the Jacobian is contained in $[-\tfrac{1}{8},0]$, leading to $\bigO(8^{-n})$ local rate of contraction. The geometric mean step additionally removes the constant $c$ that is an artefact of the usual Sinkhorn algorithm. In Theorem \ref{t:ASSA}, whose proof is in Appendix \ref{a:ASSAProof}, we show in a general setting that Algorithm \ref{alg:assa} is guaranteed to converge for any positive initial guess, and, assuming a good initial guess, converges at the $\bigO(8^{-n})$ rate with a valid stopping condition. Around $40$ ASSA iterates are typically sufficient to obtain an estimate of the Sinkhorn weight accurate to double floating point.

\begin{theorem}\label{t:ASSA}
	Suppose $\mu$ is a measure and $\D$ a positive operator that is bounded, positive semi-definite and self-adjoint on $L^2(\mu)$ and bounded on $L^\infty(\mu)$.

	Let $U$ solve the Sinkhorn problem for this operator, and let $U^{(n)}$ be the $n$th iterate of the accelerated symmetric Sinkhorn algorithm \RR{eq:ASSA1}{eq:ASSA3} with $U^{(0)} > 0$ . Then
	\begin{enumerate}[(a)]
		\item (Global convergence) For all $n \geq 0$ and $U^{(0)} > 0$,
		\[ \| \log U^{(n)} - \log U\|_{L^\infty(\mu)} \leq 2 \left(\theta\tfrac{1+\theta}{2}\right)^n \| \log U^{(0)} - \log U\|_{L^\infty(\mu)}, \]
		where $\theta<1$ is the worst-case contraction rate of standard Sinkhorn iteration, given in the proof \R{eq:ConeContractionRate}.
		\item (Local convergence rate) If $\|\log U^{(0)} - \log U\|_{L^\infty(\mu)} \leq k < 0.1$, then if $k'' := k e^{4k} (2 + \half k e^{4k}) < \tfrac{3}{8}$, the faster convergence holds
		\[ \| \log U^{(n)} - \log U\|_{L^2(\mu)} \leq (\tfrac{1}{8} + k'')^n \|\log U^{(0)} - \log U\|_{L^2(\mu)}. \] 
		\item (Stopping condition) Under the conditions of part (b),
		\[ \| \log U^{(n)} - \log U\|_{L^2(\mu)} \leq (1 - (\tfrac{1}{8} + k'')^{-1})^{-1} \|\log U^{(n)} - \log U^{(n-1)}\|_{L^2(\mu)}. \]
	\end{enumerate}
\end{theorem}
\begin{proposition} \label{p:ASSAWorks}
	The empirical measure $\rho^{(M)}$ and kernel operator $\D^{(M)}_\eps$ respectively satisfy the conditions for Theorem \ref{t:ASSA}.
\end{proposition}

Note that when $\mu$ is a discrete measure (e.g. $\mu = \rho^M$) we can recover bounds on the $L^\infty$ norm using norm equivalence:
\[ \|\cdot\|_{L^\infty(\rho^M)} \leq  M^{-1/2}\|\cdot\|_{L^2(\rho^M)}. \]
It is also possible to relax the positive semi-definiteness constraint on the kernel operator $\D$, as long as the negative spectrum of the weighted operator $\P$ is far away from $-1$.

Because the only steps in ASSA are standard Sinkhorn iteration and a geometric mean, ASSA is very well-conditioned, and can be expected to perform well in more general circumstances, including for samples on curved manifolds and from distributions with non-compact support: in Figure \ref{fig:ASSA} fast convergence of ASSA is shown for a Gaussian sampling distribution. 

As an initial value for iteration we use the standard $\alpha = \half$ right-hand weight $U^{(0)} = (\D^M_\eps 1)^{-1/2}$. According to the following proposition, when $\eps, \delta \ll 1$, this guess should be close enough to the Sinkhorn weight that the fast local convergence rate takes hold immediately.
\begin{proposition}[ASSA initialisation] \label{p:ASSAInitialisation}
	There exist constants ${{C}}_{11}, {{C}}_{12}$ independent of $M, \eps$ such that if $\delta < {{C}}_{11}$ and $\eps \leq \eps_0$, then
	\[\|\log (\D^M_\eps 1)^{-1/2} - \log U^M_\eps\|_{L^\infty(\rho^M)} \leq {{C}}_{12}(\delta + \eps),\]
	where $(\D^M_\eps 1)^{-1/2} $ is the initial condition for ASSA.
\end{proposition}


\begin{figure}
	\centering
	\includegraphics[width=0.6\linewidth]{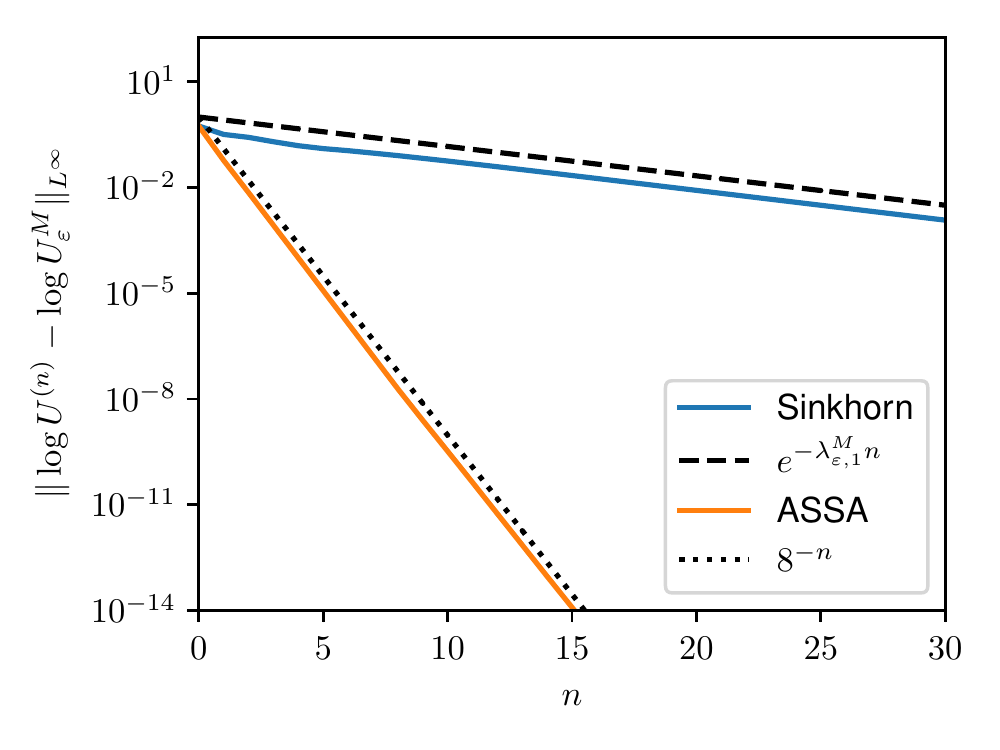}
	\caption{Convergence of standard Sinkhorn iteration (blue) and ASSA (orange) for an $M=3000$ sample from the standard normal distribution in dimension 3 with kernel parameter $\eps = 0.5$.}
	\label{fig:ASSA}
\end{figure}

These results are proven in Appendix \ref{a:ASSAProof}.


%
%

\section{Function space results}\label{s:FunctionResults}

Before we study the $\eps \to 0$ operator limit, we state some useful results in functional analysis.

Recall from Section \ref{ss:FunctionDef} that we defined scales of fractional Sobolev spaces $\W[p]{s}$ of functions $\phi$ for which $J^s \phi \in L^p$, where the sectorial Bessel operator $J := I - \Delta$.

A sufficient condition for a Banach space operator $\mathcal{A}: B \to B$ to be {\em sectorial} is that its spectrum is confined to a left open half-plane and there exists $C<\infty$ such that for $\lambda$ in the complement of this half-plane
\[ \| (\lambda + \mathcal{A})^{-1} \|_B \leq C | \lambda |^{-1}. \]
The operators $J$ and $\tilde J := I - 2\calL$ are both sectorial on $L^p = \W[p]{0}$ provided that our measure density $\rho \in C^{1+\beta}$ \citep{Haase06}. The Bessel operator $J$ is also sectorial on $\W[p]{s}$ for all positive $s$.


From Theorem 1.4.8 in \citet{Henry06} and using that $\nabla$ is bounded as an operator from $\W[p]{s+1} \to \W[p]{s}$, we have by induction that $\tilde J = I - 2L$ is a sectorial operator on $\W[p]{r},\, r \leq s,\, p >1$ and that for $\beta \in [0,1]$, $\tilde J^{\beta/2}$ is bounded as an operator $\W[p]{r+\beta}\to \W[p]{r}, r \leq s$. The condition for this to hold is that multiplication by $J^{1/2} \log \rho$ is bounded on $\W[p]{r}, r \leq s$: this is assured by the Leibniz rule for fractional derivatives $J^{\beta/2}$ \citep{Bourgain14, Li19}, provided $\rho \in \W[p]{s}$ and $s \geq 1$.

Standard results, for instance in Chapter 1 of \citet{Henry06}, and the aforementioned Leibniz rule, give the following, as well as analogues for $\check \calL_\alpha$:
\begin{proposition}\label{p:GeneratorBounds}
	Suppose that $\rho \in \W{s}$ for $s \geq 1$. Then for all $p\in(1,\infty]$:
	\begin{itemize}
		\item There exist constants $K^\nabla_{p}$ such that for all $r \geq 0$
		\[ \| \nabla \|_{\W[p]{r+1} \to \W[p]{r}},  \| \nabla \cdot \|_{\W[p]{r+1} \to \W[p]{r}} \leq K^\nabla_p; \]
		\item For all $0\leq q\leq r \leq s$ there exists a constant $K^\times_{p;r,s}$ such that for all $\phi \in \W[p]{r}, \psi \in \W{s}$,
		\[ \| \phi \psi \|_{\W[p]{r}} \leq K^\times_{p;q,r,s} \|\phi\|_{\W[p]{r}} \|\phi\|_{\W{s}}; \]
		\item For all $r < s-2$, there exists $K_{p;r}$ such that
		\[ \| \calL \|_{\W[p]{r+2} \to \W[p]{r}} \leq K_{p;r}; \]
		\item For all $s < k+\beta$, $\beta \in (0,1)$, there exists $K^{C}_{k+\beta,s}$ such that the norm of the inclusion map $C^{k+\beta} \to \W{s}$ is bounded by $K^{C}_{k+\beta,s}$.
		\item For all $q \leq r \leq s$ and all $T > 0$ there exists $K^T_{p;q,r}$ such that for $t \in [0,T]$
		\[ \| e^{t\calL} \|_{\W{q} \to \W{r}} \leq t^{-(r-q)/2} K^T_{p;q,r}; \]
		\item There exists $a>0$ such that for all $q \leq r \leq s$ and all $T > 0$ there exists $\tilde K^T_{p;q,r}$ such that for $t \in [0,T]$
\[ \| e^{t\calL} |_{Z \cap \W{q}} \|_{\W{q} \to \W{r}} \leq t^{-(r-q)/2} e^{-a t} \tilde K^T_{p;q,r},\]
where the $\calL$-invariant subspace
\begin{equation} Z = \left\{ \phi \in L^\infty : \int_\domain \phi\, \rho\, \dd x = 0 \right\}. \label{eq:ZeroSpace} \end{equation}
	\end{itemize}
\end{proposition}

\section{Convergence of Sinkhorn weights as $\eps \to 0$}\label{s:DeterministicSinkhorn}

Our convergence analysis requires an understanding the behaviour of the continuum limit Sinkhorn weights $U_\eps$. These satisfy the equation \R{eq:SinkhornProblem}, which in this section we will find useful to formulate as
\begin{equation} U_\eps^{-1} = \mathcal{C}_\eps (\sigma^2 U_\eps), \label{eq:SinkhornProblemAgain} \end{equation}
where $\mathcal{C}_\eps$ is convolution with the Gaussian kernel $g_{\eps,L}$ and $\sigma^2 := \rho$. We expect $U_\eps$ to converge to $\sigma^{-1} = \rho^{-1/2}$ as $\eps\to 0$, but because the kernel $g_{\eps,L}$ becomes singular as $\eps \to 0$ this is not trivial.

We consider this problem by formulating $U_\eps$ as the fixed point (up to constant scaling) of the Sinkhorn iteration:
\begin{align}
	U^{(n+1)} &= 1/(\mathcal{C}_\eps (\sigma^2 U^{(n)}))\label{eq:SinkhornIter1}
\end{align}
Since for fixed $\eps > 0$ the operator $\mathcal{C}_\eps \sigma^2$ is uniformly positive, Sinkhorn iteration is a contraction on the cone of positive functions and thus for all initial conditions $U_0 > 0$ the convergence holds \citep{Sinkhorn64}
\[ U^{(2n)} \to c U_\eps, U^{(2n+1)} \to c^{-1} U_\eps \]
for some $c >0$ depending on $U^{(0)}$. Note that while the iteration \R{eq:SinkhornIter1} in the $\eps \to 0$ limit has 2-periodic dynamics for all initial conditions, we do recover a fixed point $U_0 = \rho^{-1/2} = \sigma^{-1}$ that is the solution of the Sinkhorn problem \R{eq:SinkhornProblemAgain} for $\eps = 0$.

Motivated by the log-space formulation of cone metrics we set
\begin{align*} w^{n\eps} &= (-1)^{n} \log \sigma U^{(n)},
\end{align*}
so that
\begin{align} w^{(2n+1)\eps} &= \mathcal{N}_\eps w^{2n\eps} \label{eq:SinkhornLogDiscrete1}\\
w^{(2n+2)\eps} &= -\mathcal{N}_\eps(-w^{2n\eps}),\label{eq:SinkhornLogDiscrete2}
\end{align}
where the nonlinear semigroup $(\mathcal{N}_t)_{t\geq 0}$ is given by
\[ \mathcal{N}_t\phi = \log (\sigma^{-1} \mathcal{C}_t(\sigma e^{\phi})).\]
Using that $\frac{\dd}{\dd t}\mathcal{C}_t = \half \Delta \mathcal{C}_t$ it is straightforward to show that the infinitesimal generator of $\mathcal{N}_t$ is given by
\[ \left.\frac{\dd}{\dd t} \mathcal{N}_t\phi\right|_{t=0} = \half \Delta \phi + \half | \nabla \phi |^2 + \frac{\nabla \sigma}{\sigma} \cdot \nabla \phi + \frac{\Delta \sigma}{2\sigma}. \]
By using $\mathcal{N}_t$ to interpolate \RR{eq:SinkhornLogDiscrete1}{eq:SinkhornLogDiscrete2} in time, we can thus write Sinkhorn iteration as a nonlinear PDE
\begin{equation}
\partial_t w^t = 
\calL w^t + (-1)^{\lfloor \eps^{-1} t \rfloor}\left(\half | \nabla w^t |^2 + \frac{\Delta \sigma}{2\sigma}\right).
\label{eq:SinkhornPDE}
\end{equation}
This reformulation can be seen as the reverse of the Cole-Hopf transformation \citep{Evans98}.

Now, the the PDE \R{eq:SinkhornPDE} can be decomposed as a sum of an autonomous linear part, 
in fact the limiting generator of the diffusion maps problem $\calL$, with a non-autonomous, rapidly oscillating nonlinear part that has time integral zero. Consequently, we can apply averaging results to this system as $\eps \to 0$. This will give us convergence of $w^t$ and thus $U_\eps$:



\begin{theorem}\label{t:UConvergence}
	Suppose $\rho \in \W{s},\, s \geq 2$.

	Then the PDE \R{eq:SinkhornPDE} has a unique limit cycle $w_\eps^t$  with $w_\eps^{t+\eps} = -w_\eps^t$ and $w_\eps^0 = \log \rho^{1/2} + \log U_\eps$.

	Furthermore for all $0 \leq r < s +1$,
	\[ \lim_{\eps \to 0} \sup_t \| w_\eps^t \|_{\W{r}} = 0, \]
	and
	\[ \lim_{\eps \to 0} \| \log U_\eps - \log \rho^{-1/2} \|_{\W{r}} = 0. \]
\end{theorem}

This has the following immediate corollary:
\begin{corollary}\label{c:WeightBound}
	Suppose $\rho \in \W{s},\, s \geq 2$. Then there exists a constant ${{C}}_{22}$ such that for all $\eps \leq \eps_0$
	\[ \sup_{\eps \leq \eps_0} \| U_\eps \|_{C^3} \leq {{C}}_{22} < \infty \]
	and for all $r < s+1$ a constant ${{C}}_{23,r}$ such that that for all $\eps \leq \eps_0$
	\[ \sup_{\eps \leq \eps_0} \sup_{t} \| w_\eps^t \|_{\W{r}} \leq {{C}}_{23,r} < \infty. \]
\end{corollary}

The uniform bounds on the ($2\eps$-periodic) limit cycle $w^t_\eps$ are of particular use to us, because $w^t_\eps = (-1)^{\lfloor t/\eps\rfloor} w^t_\eps$: that is, up to a periodic change of sign, it is the same as the $\eps$-periodic drift error term in the time discretisation of diffusion maps \R{eq:SDMPDE}.

\begin{remark}
By applying instead Theorem 1.1 of \citet{Ilyin98}, one can show that as $\eps \to 0$, the solution of the Sinkhorn iteration PDE \R{eq:SinkhornPDE}, $w^t$, converges to an averaging limit
\[ \partial_t \bar w^t = \calL \bar w^t \]
over finite time scales (c.f. the Monge-Ampere PDE derived for non-symmetric Sinkhorn iteration in \citet{Berman17}). As a result, one recovers the asymptotic rate of (standard) Sinkhorn iteration
\[ \lim_{n\to\infty} \frac{-\log\|U^{(n)} - U_\eps\|}{n} = -\lambda_1 \eps, \]
where $-\lambda_1$ is the first non-zero eigenvalue of the Langevin dynamics $\calL$. 
\end{remark}

\begin{proof}[Proof of Theorem \ref{t:UConvergence}]
	This amounts to checking the conditions of Theorem 1.2 of \citet{Ilyin98}. Due to the invariance of constant functions under Sinkhorn iteration we will project our PDE \R{eq:SinkhornPDE} onto the subspace of zero mean functions $Z$ defined in \R{eq:ZeroSpace}. We thus consider
	\begin{equation} \partial_t w^t = 
	\calL w^t + \mathcal{F}(w^t, \eps^{-1} t) + \mathcal{X}(\eps^{-1} t),
	\label{eq:SinkhornPDEProj} \end{equation}
	where
	\begin{align*}
	\mathcal{F}(\phi,\tau) &= (-1)^{\lfloor \tau \rfloor} \half (I-\mathcal{Z})\left(| \nabla \phi |^2\right) \\
	\mathcal{X}(\tau) &= (-1)^{\lfloor \tau \rfloor} (I-\mathcal{Z})\left(\frac{\Delta \sigma}{2\sigma}\right),
	\end{align*}
	and the projection operator
	\[(\mathcal{Z}\phi)(x) := \int_\domain \phi(y)\rho(y)\,\dd y.\]

	 Suppose $r \geq s-1$ (the result will then follow immediately for $r < s-1$). Set Banach spaces $E = \W{r} \cap Z, F = \W{r-1} \cap Z, X = \W{s-2} \cap Z$ and $\mathcal{E} = Z$.

	Proposition \ref{p:GeneratorBounds} implies the various conditions on the linear operator $\calL$ and the averaged semigroup $e^{\calL t}$ required for Theorem 1.2 of \citet{Ilyin98}. We also have that the nonlinear part $\mathcal{F}: E \times \mathbb{R} \to F$ is Lipschitz on bounded subsets of $E$ and the driver $\mathcal{X}$ has range in $X$. Both are locally integrable over $\tau$. As a result, we have that the attractor of \R{eq:SinkhornPDEProj} converges in the strong space $E$ uniformly to the attractor of $\partial_t w^t = \calL w^t$ in $E$, i.e.~zero. In other words, if $\{w_\eps^{t,\mathcal{Z}}\}_{t \in \mathbb{R}}$ is this attractor (which by the convergence of Sinkhorn iteration is necessarily a unique limit cycle), then
	\[ \lim_{\eps \to 0} \sup_t \| w_\eps^{t,\mathcal{Z}} \|_E = 0. \]

	If we let $w_\eps^t$ be the solution of the unprojected PDE \R{eq:SinkhornPDE} corresponding to the true Sinkhorn weights with $w_\eps^{n\eps} = (-1)^n \log \sigma U_\eps$, then for all $t$ one has $w_\eps^{t+\eps} = -w_\eps^t$; furthermore if $w_\eps^{t,\mathcal{Z}}$ is the attractor (necessarily a limit cycle) of the projected PDE \R{eq:SinkhornPDEProj} then
	\[ w_\eps^t - w_\eps^{t,\mathcal{Z}} = \mathcal{Z} w_\eps^t = \int_\domain w_\eps^t\,\rho\,\dd y. \]
	From \R{eq:SinkhornPDE} and using that $\nabla w_\eps^t = \nabla w_\eps^{t,\mathcal{Z}}$ we find that
	\[ \sup_{\eps \leq \eps_0} \left\| \partial_t \mathcal{Z} w_\eps^t \right\| = \sup_{\eps \leq \eps_0} \left| \int_{\domain} \left(\half | \nabla w_{t,\mathcal{Z}} |^2 + \frac{\Delta \sigma}{2\sigma}\right)\, \rho\, \dd x \right| < \infty. \]
	Then using that $\mathcal{Z} w_\eps^{t + \eps} = \half(\mathcal{Z} w_\eps^{t + \eps} - \mathcal{Z} w_\eps^{t})$ implies that
	\[ \lim_{\eps \to 0} \sup_{t} \|w_\eps^t - w_\eps^{t,\mathcal{Z}}\|_{\W{r}} \to 0, \]
	giving us what is required.
\end{proof}

Because $w^t_\eps$ is, up to a time-varying change of sign, the drift term in the temporally-discretised PDE \R{eq:SDMPDE}, we will find it useful to make some more specific estimates on $w^t_\eps$ to prove the operator convergence in the next section. In particular, we will show that $w^t_\eps = \bigO(\eps)$, and that $w_\eps^t$ is, up to $\bigO(\eps^2)$, symmetric in time.
\begin{lemma}\label{l:SinkhornDriftBound}
	Suppose $\rho \in \W{s}$, $s > 2$ and $w^t_\eps$ is as in Theorem \ref{t:UConvergence}. Then for all $r \in [2,s+1)$ there exist ${{C}}_{24,r},{{C}}_{25,r}$ such that for all $\eps\leq\eps_0$, $t \in [0,\eps]$
	\begin{equation} \| w^t_\eps \|_{\W{r-2}} \leq \half {{C}}_{25,r}\eps\label{eq:WBound} \end{equation}
	and
	\begin{equation} \| J^{-r^*/2}(w^t_\eps + w^{\eps-t}_\eps) \|_{\W{r-4+{r^*}}} \leq {{C}}_{24,r} \eps^2, \label{eq:WSymBound} \end{equation}
	where $r^* = \max\{4-r,0\}$.
\end{lemma}

\begin{proof}[Proof of Lemma \ref{l:SinkhornDriftBound}]
	Making use of Corollary \ref{c:WeightBound} and Proposition \ref{p:GeneratorBounds}, we have that
	\[ \|\partial_t w^t_\eps\|_{\W{r-2}} \leq \half K_{\infty;r-2} {{C}}_{23,r} + K^\times_{\infty;,r-2,r-1,r-1} (K^\nabla_\infty {{C}}_{23,r-1})^2 + \left\| \frac{\Delta \rho^{1/2}}{2\rho^{1/2}} \right\|_{\W{r-2}} =: {{C}}_{25,r}. \]

	From Theorem \ref{t:UConvergence} we have that $w^t_\eps = -w^{t-\eps}_\eps$, and so as a result
	\[ \sup_{t\in[0,\eps]}\|w^t_\eps\|_{\W{r-2}} =  \sup_{t\in[0,\eps]}\half \|w^t_\eps - w^{t-\eps}_\eps\|_{\W{r-2}} \leq \half {{C}}_{25,r}\eps, \]
	as required for \R{eq:WBound}.

	To obtain \R{eq:WSymBound}, we will want to take the second derivative in time: however, for $r\in(2,4)$ we do not have enough regularity in our function spaces to do that, so we will introduce an inverse fractional derivative $J^{-r^*/2}$ to compensate. In particular, we have that for $t \in (0,\eps)$,
	\begin{align*} \partial_{tt} J^{-r^*/2} w^t &= \partial_t J^{-r^*/2} \partial_t w^t\\
	 &= \half \partial_t J^{-r^*/2} (- J w^t + w^t + \nabla \log \rho \cdot \nabla w^t + \nabla \partial_t w^t \cdot \nabla w^t + \Delta \sigma/\sigma) \\
	&= -\half J^{1-r^*/2}\partial_t w^t + \half J^{-r^*/2} \partial_t w^t + \half J^{-r^*/2} \nabla (\log \rho + 2w^t) \cdot \nabla \partial_t w^t, \end{align*}
	and as a result this second time derivative is uniformly bounded in $\W{r-4+r^*}$:
	\begin{align*}\| \partial_{tt} J^{-r^*/2} w^t\|_{\W{r-4+r^*}} &\leq \half {{C}}_{25,r} + \half {{C}}_{23,r} + \\
	&\qquad \half K^\times_{\infty;r-4+r^*,r-1,r-3} (K^\nabla_\infty)^2 {{C}}_{25,r} (\|\log\rho\|_{\W{r-1}}+2{{C}}_{23,r-1}) \\
	& =: \half {{C}}_{24,r}. \end{align*}

	Thus, by applying Taylor's theorem,
	\begin{equation} \left\| J^{-r^*/2} \left(w^t_\eps + w^{\eps-t}_\eps - 2w^{\eps/2}_\eps\right) \right\|_{\W{r-4+r^*}} \leq \sup_{t \in [0,\eps]} \| \partial_{tt} J^{-r^*/2} w^t \|_{\W{r-4+r^*}} \eps^2 \leq \half {{C}}_{24,r} \eps^2. \label{eq:SinkhornWeightTaylor} \end{equation}
	Since $w^0_\eps = -w^\eps_\eps$, by setting $t = 0$ in \R{eq:SinkhornWeightTaylor} we have
	\begin{equation} \|2J^{-r^*/2} w^{\eps/2}_\eps\|_{\W{r-4+r^*}} \leq \half {{C}}_{24,r} \eps^2. \label{eq:SinkhornWeightMidpoint} \end{equation}
	Recombining this with \R{eq:SinkhornWeightTaylor} we obtain the necessary result.
\end{proof}

\begin{remark}
Since from \R{eq:SinkhornWeightMidpoint} we have for $s>4$ (i.e.~$\rho \in C^{4+\beta}$) that
\[ \|\D_{\eps/2}U_\eps - \rho^{1/2}\|_{L^\infty} = \bigO(\eps^2), \]
the Sinkhorn problem can be used to perform second-order non-parametric estimation on the density $\rho$.
\end{remark}

\section{Deterministic convergence of operators}\label{s:DeterministicConvergence}

Recall from \R{eq:SemigroupDeterministic} that the deterministic approximation to the semigroup is
\[ \P_\eps = U_\eps \D_{\eps} U_\eps. \]
In this section, we will harness our results on the Sinkhorn weight $U_\eps$ from the previous section to Theorem \ref{t:DeterministicOperator} on convergence of $\P_\eps$ to the semigroup $e^{\eps \calL}$. Before this, we will make some remarks on the rate of convergence to the semigroup.
\begin{remark}\label{r:GoodBecauseSymmetric}
For the Sinkhorn normalisation the bias error convergence is of second order in the timestep $\eps$, unlike the first-order convergence for standard weights (c.f. Proposition \ref{p:DeterministicOperatorStandard}). This is actually a result of the self-adjointness of the normalised operator.

To be more specific (and to outline the strategy of the proof of Theorem \ref{t:DeterministicOperator}), we can write the action of $\P_\eps$ as solving the PDE \R{eq:SDMPDE}, which we recall here,
\begin{align}
\phi^0 &= \phi \notag\\
\partial_t \phi^t &= \calL \phi^t + \nabla \hat w^t_\eps \cdot \nabla \phi^t, \label{eq:GoodBecauseSymmetricPDE}
\end{align}
so that $\phi^\eps = \P_{\eps} \phi$, where recall that the discrepancy in drift compared with the semigroup is
\[ \hat w^t_\eps = \log (\D_t U_\eps) - \half \log \rho,\, t \in [0,\eps).\]
Note that because the Sinkhorn normalisation \R{eq:SinkhornProblem} is required to be symmetric,
\[ \hat w^0_\eps = \log U_\eps + \half \log \rho = -\lim_{t\uparrow\eps}\hat w^{t}_\eps.\]

If $\phi, \rho, U_\eps$ are of sufficiently high regularity, for small $\eps$ we can average the PDE \R{eq:GoodBecauseSymmetricPDE} over $t \in [0,\eps]$:
\begin{align*}
\partial_t \phi^t \approx \calL \phi^t + \nabla  \left(\eps^{-1}\int_0^\eps \hat w^t_\eps\, \dd \tau\right) \cdot \nabla \phi^t.
\end{align*}
The averaged drift can then be approximated using the trapezoidal rule with
\begin{align*}
\eps^{-1} \int_0^\eps \hat w_\tau^\eps\, \dd \tau &= \half (\hat w_\eps^0 + \lim_{t\uparrow\eps}\hat w^{t}_\eps) + \bigO(\eps^2) = \bigO(\eps^2).\end{align*}
As a result the operator $\P_{\eps}$ should closely approximate $e^{\eps \calL}$, as required.
\end{remark}

\begin{remark}\label{r:SinkhornIsBest}
The $\bigO(\eps^3)$ rate of convergence $\P_\eps \to e^{\eps\calL}$ is in general the best possible for operators of the form $V_\eps \D_{\eps} U_\eps$. We can best see this by comparing the re-weighted operators for $n=1$:
\[ \sigma V_\eps \D_{\eps} U_\eps \sigma^{-1} =  V_\eps\sigma e^{\eps \half \Delta} \sigma U_\eps\]
and
\[ \sigma e^{\eps \calL} \sigma^{-1} = e^{\eps (\half \Delta - \half \sigma^{-1} \Delta \sigma)}, \]
where $\sigma = \rho^{1/2}$.
Taking a power series in $\eps$ and writing each side in the form
\[ \sum_{k \geq 0} \left( \tfrac{1}{2^k} \Delta^k + \sum_{j=0}^{k-1} \left(\beta_{j,k} \Delta^j + \nabla \beta_{j,k} \cdot \nabla \Delta^{j-1}\right) \right) \eps^k, \]
we see that for the $\beta_{k-1,k}$ coefficients to match it is necessary that
\[ \left.\frac{\partial}{\partial \eps} (V_\eps + U_\eps) \right|_{\eps = 0} = - \frac{1}{2(k-1)!} \sigma^{-2} \Delta\sigma. \]
Unless $\Delta\sigma \equiv 0$, i.e.~$\sigma = \rho^{1/2}$ is constant, then this can only hold simultaneously for $k = 1,2$: an $\bigO(\eps^3)$ error between $e^{\eps\calL}$ and $V_\eps \D_{\eps} U_\eps$ is thus the best possible (and hence we expect also an $\bigO(\eps^2)$ error for the spectral data).
\end{remark}

To prove Theorem \ref{t:DeterministicOperator}, we will require the following result:
\begin{proposition}\label{p:SchauderC3}
	Suppose $\rho \in \W{s},\, s > 2$. Then for all $T>0$, $\beta \in (0,\min\{s-2,1\})$, there exists a constant ${{C}}_{31,T,\beta}$ such that for all $|t_1-t_0|\leq T$,
	\[  \| S_\eps(t_1,t_0) \|_{C^{3+\beta}} \leq {{C}}_{31,T,\beta}. \]
\end{proposition}
\begin{proof}[Proof of Proposition \ref{p:SchauderC3}]
	From Corollary \ref{c:WeightBound}, we have for all $r<s+1$ an $\eps$-uniform bound on the $\W{r}$ norm of $w^t_\eps = (-1)^{\lfloor t/\eps \rfloor} \hat w^t_\eps$. We therefore also have uniform in $\eps$ bounds on the $C^{3+\beta}$ norm of $\half \log \rho + \hat w^t_\eps$ for $2+\beta < s$. We can thus apply Theorem 1.2 of \citet{Lorenzi00} to (\ref{eq:SDMPDE}) to obtain relevant uniform bounds on $\| S_\eps(t_1,t_0) \|_{C^{2+\beta}}$. By observing that (\ref{eq:SDMPDE}) implies that
	\begin{equation} \partial_t \partial_{x_i} \phi^t = \calL \partial_{x_i} \phi^t + \nabla \partial_{x_i}(\half \log \rho + w^t) \cdot \nabla \phi^t, \label{eq:SDMPDEDerivative} \end{equation}
	we can then re-apply \citet{Lorenzi00} to obtain bounds on $\| S_\eps(t_1,t_0) \|_{C^{3+\beta}}$.
%
\end{proof}


We now prove Theorem \ref{t:DeterministicOperator}.
\begin{proof}[Proof of Theorem \ref{t:DeterministicOperator}]
	The definitions of $\J_\eps, \K_\eps, \P_\eps, \M_{\eps,n}$ follow immediately by observing that
	\[ S_\eps(t_1,t_0) \phi = (\D_{t_1} U_\eps)^{-1}\C_{t_1-t_0}((\D_{t_0} U_\eps) \phi )\]
	for $0\leq t_0 < t_1 \leq \eps$.
	\\

	Writing $S_0(t_1,t_0) = e^{(t_1-t_0)\calL}$, the discrepancy in the errors is
	\begin{equation} S_\eps(t_1,t_0) - S_0(t_1,t_0) = \int_{t_0}^{t_1} S_0(t_1,\tau) \nabla \hat w_\eps^{\tau} \cdot \nabla S_\eps(\tau,t_1)\,\dd\tau. \label{eq:SolutionOperatorDifference} \end{equation}
	We can then bound
	\begin{align*} \|S_\eps(t_1,t_0) - S_0(t_1,t_0)\|_{C^{3+\beta} \to L^\infty} &\leq (t_1-t_0) \sup_{\tau \in [0,\eps]} \left(\|S_0(t_1,\tau)\|_{L^\infty}  K^\times_{\infty;0,0,1} K^\nabla_\infty \|\hat w_\eps^{\tau}\|_{\W{1}}\times\right.\\
	&\qquad  \left.K^\nabla_\infty \|S_\eps(\tau,t_0)\|_{C^{3+\beta} \to \W{1}}\right)\\
	& \leq \eps\, K^\times_{\infty;0,0,1} (K^\nabla_\infty)^2 {{C}}_{25,1} \eps\, \|S_\eps(\tau,t_0)\|_{C^{3+\beta} \to \W{1}}\eps\\
	& \leq \eps\, K^\times_{\infty;0,0,1} (K^\nabla_\infty)^2 {{C}}_{25,1} K^{C}_{3+\beta,1} {{C}}_{31,T,\beta} \eps^2, \end{align*}
	where in the second-last inequality we used that $\hat w^t_\eps = 
	(-1)^{\lfloor t/\eps\rfloor} w^t_\eps$, and then Lemma \ref{l:SinkhornDriftBound}, and in the last inequality we used Proposition \ref{p:SchauderC3}.

	Using that $S_\eps(t_1,t_0) L^\infty \subset C^0$ and that the $C^{3+\beta}$ norm dominates the $\W{1}$ norm, we obtain \R{eq:GeneratorBound} for $t_1-t_0 < \eps$.
	\\
	
	We can use this result to reduce from all $0<t_1-t_0<T$ to the case where $t_1-t_0$ is a multiple of $\eps$: mathematically, this is because if $m = \lfloor (t_1-t_0)/\eps \rfloor$, then we have
	\begin{align*} &\|S_\eps(t_1,t_0) - S_0(t_1,t_0)\|_{C^{3+\beta} \to L^\infty} \\
	&\qquad\leq \|S_\eps(t_1,t_0+m\eps)\|_{L^\infty} \|S_\eps(t_0+m\eps,t_0) - S_0(t_0+m\eps,t_0)\|_{C^{3+\beta} \to L^\infty} \\
	&\qquad\qquad + \|S_\eps(t_1,t_0+m\eps) - S_0(t_1,t_0+m\eps)\|_{C^{3+\beta}\to L^\infty}  \|S_0(t_0+m\eps,t_0)\|_{C^{3+\beta}}\\
	&\qquad \leq \|S_\eps(t_0+m\eps,t_0) - S_0(t_0+m\eps,t_0)\|_{C^{3+\beta} \to L^\infty} \\
	&\qquad\qquad+ K^T_{\infty;3,3} \|S_\eps(t_1,t_0+m\eps) - S_0(t_1,t_0+m\eps)\|_{C^{3+\beta} \to L^\infty}.  \end{align*}
	 At the same time, simply applying the previous argument to $t_1-t_0 = m\eps$ will give an error of size $\bigO(\eps)$ instead of $\bigO(\eps^2)$: we need to average over a cycle of $w^t_\eps$. The aim is to move all the $\nabla \hat w_\eps^{\tau} \cdot \nabla$ drift operators in \R{eq:SolutionOperatorDifference} in a period of $\hat w_\eps^{\tau}$ to the same point in time, and show that their average is small (c.f. \citet{Ilyin98}, Chapter 7 of \citet{Henry06}).

	To move the drift operators in time we will use that
	\begin{align*} \left\| \frac{d}{d\tau} S_\eps(\tau,t_0)\right\|_{C^{3+\beta} \to \W{1}} &= \| \calL + \nabla w_\eps^{\tau} \cdot \nabla \|_{\W{3} \to \W{1}} K^{C}_{3+\beta,3} \|S_\eps(\tau,t_0)\|_{C^{3+\beta}} \\
	&\leq (K_{\infty;3} + K^\times_{\infty;1,1,1} {{C}}_{25,2} (K^\nabla_\infty)^2 \eps_0) K^{C}_{3+\beta,3} {{C}}_{31,T,\beta}
	:= {{C}}_{32,T,\beta},\end{align*}
	so that if $\bar t = s+\eps\lfloor \eps^{-1}(\tau-t_0)\rfloor$,
	\[  \| S_\eps(\tau,t_0) - S_\eps(\bar t,t_0) \|_{C^{3+\beta} \to \W{1}} \leq {{C}}_{32,T,\beta} \eps, \]
	and so using Lemma \ref{l:SinkhornDriftBound},
	\[ \| S_0(t_1,\tau) \nabla \hat w_\eps^{\tau} \cdot \nabla(S_\eps(\tau,t_0) - S_\eps(\bar t,t_0)) \|_{C^{3+\beta} \to L^\infty} \leq (K^\nabla_\infty)^2 {{C}}_{25,1} {{C}}_{32,T,\beta} \eps. \]

	To change the length of the $S_0$ part, we use that, for any $r \in (3,\min\{s+1,4\})$,
	\begin{align*} \left\| \frac{d}{d\tau} S_0(t_1,\tau)\right\|_{\W{r-2} \to L^\infty} &=  \| \calL \|_{\W{2} \to L^\infty} \|S_0(t_1,\tau)\|_{\W{r-2} \to \W{2}} \\
	& \leq K_{\infty;0} (t_1-\tau)^{r/2-2} K^T_{\infty;r-2,2}, \end{align*} 
	so by integrating we have
	\[ \| S_0(t_1,\tau) - S_0(t_1,\bar t)\|_{\W{r-2} \to L^\infty} \leq K_{\infty;0} K^T_{\infty;r-2,2} (\tfrac{r}{2}-1)^{-1} \left((t_1 - \bar t)^{r/2-1} - (t_1 - \eps - \bar t)^{r/2-1}\right) \eps. \]
	We can then bound the remaining part as
		\[ \|  \nabla \hat w_\eps^{\tau} \cdot \nabla S_\eps(\bar t,t_0)\|_{C^{3+\beta} \to \W{r-2}} \leq {{C}}_{31,T,\beta} K^{C}_{3+\beta,3} K^\times_{\infty;r-2,r-2,2} (K^\nabla_\infty)^2 {{C}}_{25,r-2}. \]

	As a result, for some constant ${{C}}_{33,T,\beta}$ we have
	\begin{align*} & \|S_0(t_1,\tau) \nabla \hat w_\eps^{\tau} \cdot \nabla S_\eps(\tau,t_1) - S_0(t_1,\bar t) \nabla \hat w_\eps^{\tau} \cdot \nabla S_\eps(\bar t,t_1)\|_{C^{3+\beta} \to L^\infty}\\
	&\qquad \leq {{C}}_{33,T,\beta} \left((t_1 - \bar t)^{r/2-1} - (t_1 - \eps - \bar t)^{r/2-1}\right) \eps, \end{align*}
	and so using \R{eq:SolutionOperatorDifference},
	\begin{align} &\|S_\eps(t_0+m\eps,t_0) - S_0(t_0+m\eps,t_0)\|_{C^{3+\beta}\to L^\infty} \notag \\
	&\qquad \leq {{C}}_{33,T,\beta} \max\{m\eps,(m \eps)^{r/2-1}\} \eps^2 + \sum_{n=0}^{m-1}  \left\| e^{(m-n)\eps\calL} \nabla \bar w_\eps \cdot \nabla S_\eps(t_0+n\eps,t_0)\right\|_{C^{3+\beta} \to L^\infty}, \label{eq:AveragingSums} \end{align}
	where
	\begin{equation*}
	\bar w_\eps := \int_{t_0}^{t_0+\eps} w_\eps^t\,\dd t.
	\end{equation*}
	Then, using that
	\[ \int_{t_0}^{t_0+\eps} \hat w_\eps^t\,\dd t = \int_0^\eps w_\eps^t\, \dd t = \int_0^{\eps/2} (w_\eps^t + w_\eps^{\eps-t})\, \dd t,  \]
	we have from Lemma \ref{l:SinkhornDriftBound} that, for $r \in (3,\min\{4,s\})$,
	\[ \| J^{-(4-r)/2} \bar w_\eps \|_{L^\infty} \leq \half {{C}}_{24,r} \eps^3. \]
	This means that $\bar w_\eps$, a function in $\W{r}$, is particularly small in the negative Sobolev norm $\W{r-4}$. To avoid dealing with negative Sobolev spaces (which are complex to negotiate particularly due to the endpoint parameter of integrability $p = \infty$), we make an excursion into spaces associated with $\infty > p \gg 1$, where we can easily apply dual norms to get the result we would like.

	If we let $p \in (1,\infty)$ and set $q^{-1} = 1 - p^{-1}$, then we have that for $\phi \in \W{3}$,
	\[ \left\| e^{(m-n)\eps \calL/2} \nabla \bar w_\eps \cdot \nabla \phi \right\|_{L^p} = \sup_{\|\psi\|_q = 1} \int_{\domain} \psi e^{(m-n)\eps \calL} \nabla \bar w_\eps \cdot \nabla \phi\, \dd x.
	\]
	Since $e^{(m-n)\eps \calL/2} = e^{\half (t-\bar t)\calL}$ is a symmetric kernel operator with respect to the measure $\rho\,\dd x$, this and integration by parts give that
	\begin{equation} \int_{\domain} \psi e^{(m-n)\eps \calL/2} \nabla \bar w_\eps \cdot \nabla\, \dd x = - \int_{\domain} \bar w_\eps g \, \dd x,
	\label{eq:DualBound} \end{equation}
	where
	\[g := \nabla \cdot \left( (e^{(m-n)\eps \calL}\rho^{-1}\psi)\,\rho \nabla \phi\right).\]
	This term can be bounded in the $\W[q]{4-r}$ norm with liberal use of Proposition \ref{p:GeneratorBounds}, by using that
	\begin{align*} &\left\|\nabla \cdot \left(\left(e^{(m-n)\eps \calL/2} \rho^{-1}\psi\right) \rho\, \nabla \phi\right)\right\|_{\W[q]{4-r}} \\&\qquad \leq K^\nabla_q K^\times_{q;5-r,5-r,2} \left\| \left(e^{(m-n)\eps \calL/2}\rho^{-1}\psi\right) \rho \right\|_{\W[q]{5-r}} K^\nabla_\infty \|\phi\|_{\W{3}}
	 \end{align*}
	 and
	 \[ \| \left(e^{(m-n)\eps \calL/2}\rho^{-1}\psi\right) \rho \|_{\W[q]{5-r}} \leq K^\times_{q;5-r,5-r,3} \|\rho\|_{\W{2}} K^T_{q;0,5-r} ((m-n)\eps/2)^{-(5-r)/2} \|\rho^{-1}\|_\infty. \]
	 Thus, there exist constants ${{C}}_{34,T,p,r}$ such that for all $(m-n)\eps \leq T$,
	 \[ \| g \|_{\W[q]{4-r}} \leq {{C}}_{34,T,p} ((m-n)\eps)^{-(5-r)/2}. \]
	 Returning to \R{eq:DualBound}, we obtain that
	 \[ \int_{\domain} \bar w_\eps g \, \dd x = \int_{\domain} (J^{-(4-r)/2}\bar w_\eps) (J^{(4-r)/2}g) \, \dd x, \]
	 using firstly that $J^{-(r-2)/2} J J^{(4-r)/2}$ is the identity, secondly that from \R{eq:FractionalPower}, $J^{-(r-2)/2}$ is a symmetric kernel operator, and finally integration by parts. Using this we can deduce that
 	\begin{align*} \int_{\domain} \bar w_\eps g \, \dd x &\leq \|J^{-(4-r)/2}\bar w_\eps\|_{L^p} \|g\|_{\W[q]{4-r}} \\
 		&\leq |\domain|^{1/p} \|J^{-(4-r)/2}\bar w_\eps\|_{L^\infty} {{C}}_{34,T,p,r} ((m-n)\eps)^{-(5-r)/2} \\
 		 &\leq L^{d/p} \half {{C}}_{24,r} \eps^3 {{C}}_{34,T,p,r} ((m-n)\eps)^{-(5-r)/2} \\
 		 &:= {{C}}_{35,T,p,r} ((m-n)\eps)^{-(5-r)/2} \eps^3.
 	\end{align*}

 	As a result, we can say that
 	\begin{equation}
 	  \left\| e^{(m-n)\eps \calL/2} \nabla \bar w_\eps \cdot \nabla S_\eps(t_0+n\eps,t_0)\right\|_{\W{3} \to L^p} \leq {{C}}_{35,T,p,r} ((m-n)\eps)^{-(5-r)/2} \eps^3.
 	  \label{eq:AveragingNablaPart}
 	\end{equation}
 	To obtain \R{eq:AveragingSums} from \R{eq:AveragingNablaPart}, it only remains to bound the rest of the action of the semigroup $\| e^{(m-n)\eps \calL/2}\|_{L^p \to L^\infty}$. Recalling the definition of the Gaussian kernel \R{eq:GaussianKernel}, we have the Gaussian upper estimate \citep{Liskevich00} that for $t \leq T$, and some ${{C}}_{36}, {{C}}_{37}$ depending on $\| \nabla \log \rho\|_{L^\infty}, L, d, T$,
 	\[ (e^{t\calL/2}\phi)(x) \leq {{C}}_{36} \int g_{{{C}}_{37}t}(x-y) \phi(y) \dd y. \]
 	This gives us that
	\[ \|e^{t\calL/2} \phi \|_{L^\infty} \leq {{C}}_{36} \| g_{{{C}}_{37}t/2} \|_{L^q} \|\phi\|_{L^p} \leq {{C}}_{36} q^{-d/2q} ({{C}}_{37} t/2)^{-d/p} \|\phi\|_{L^p} := {{C}}_{38,p} t^{-d/p} \|\phi\|_{L^p}. \]

	Then, applying also Proposition \ref{p:SchauderC3} for the norm of $S_\eps(t_0+n\eps,t_0)$, we have
	\begin{align*}
	&\left\| e^{(m-n)\eps\calL} \nabla \bar w_\eps \cdot \nabla S_\eps(t_0+n\eps,t_0)\right\|_{\W{3} \to L^\infty} \notag \\
	&\qquad \leq {{C}}_{38,p} {{C}}_{35,T,p,r} K^{C}_{3+\beta,3} {{C}}_{31,T,\beta} ((m-n)\eps)^{-(5-r)/2-d/p} \eps^3.
	\end{align*}
	Fixing $r$ and choosing $p > 2d/(r-3)$ we have $(5-r)/2+d/p < 1$, and thus there exists a constant ${{C}}_{36,T,\beta}$ such that for $m \eps \leq T$,
	\[ \sum_{n=0}^{m-1}  \left\| e^{(m-n)\eps\calL} \nabla \bar w_\eps \cdot \nabla S_\eps(t_0+n\eps,t_0)\right\|_{\W{3} \to L^\infty} \leq {{C}}_{36,T,\beta} \eps^2. \]
	Combining this with \R{eq:AveragingSums} gives us \R{eq:GeneratorBound} for $t_1 = t_0+m\eps$ as required.
\end{proof}

\section{Convergence of kernel operator in finite data approximation}\label{s:Particle}

We now turn to the ``variance'' error, i.e. the convergence of the finite data approximation as the sample size $M \to \infty$. In this section we begin by showing the convergence of the discretised Gaussian kernel $\D^M_\eps$ to the continuum limit $\D_{\eps}$. We prove convergence first pointwise for fixed functions, then extend to convergence in norm on fixed functions, and then 
finally to norm-convergence of operators.

Recall from \RR{eq:DMeps}{eq:Deps} that we defined the operators $\D_\eps$ and $\D^M_\eps$ as
\begin{equation} (\D_\eps \phi)(x)= \int g_{\eps,L}(x-y) \phi(y) \rho(y) \,\dd y \label{eq:DepsAgain} \end{equation}
and
\begin{equation} (\D^M_\eps \phi)(x) = \frac{1}{M} \sum_{i=1}^M g_{\eps,L}(x-x^i) \phi(x^i). \label{eq:DMepsAgain} \end{equation}
Since the $\{x_i\}$ are sampled from the measure $\rho$, the continuous operator $\D_\eps \phi$ is the expectation of the discretised operator $\D^M_\eps \phi$ with respect to this sampling. Because the discretised operator is the sum of independent random variables $g_{\eps,L}(x-x^i) \phi(x^i)$, it is therefore natural to try to construct central limit theorems.

The basic result we will use for this purpose is the following Bernstein inequality that provides strong quantitative control on the tail probabilities:
\begin{proposition}\label{p:CLTbounds}
	Consider an {\em i.i.d.} collection of bounded, centred random variables $X_i$. Then if $\E[X^2] \leq \nu$ and $c \leq 6 \nu / \| X\|_0$,
	\[ \Pr\left(\left| \frac{1}{M} \sum_{i=1}^M X_i \right| > c \right) \leq 2 e^{-M c^2/6\nu}. \]
\end{proposition}


To deal with the fact that we are using the periodised Gaussian kernel $g_{\eps,L}$ rather than the standard one $g_\eps$, we will require the following proposition:
\begin{proposition}\label{p:GaussianL}
	Define the increasing functions of $\eps$
		\begin{align*} \gamma_{\eps,L} &= \sum_{j \in \mathbb{Z}} e^{-j^2 L^2/2\eps},\\
		\gamma_{\eps,L}' &= \sum_{j=1}^\infty (2j+1) L \eps^{-1/2}  e^{-(2j-1)^2 L^2/8\eps}.
		\end{align*}
	Then for all $x \in [-L/2,L/2]^d$,
	\begin{align*}
		g_{\eps,L}(x) &\leq (1+\gamma_{\eps,L})^d g_\eps(x), \\
		 \sup g_{\eps,L} &\leq \gamma_{\eps,L}^d g_\eps(0), \\
		\Lip g_{\eps,L} &\leq \gamma'_{L,\eps,d} \Lip g_\eps,
	\end{align*}
	where
	\[\gamma'_{L,\eps,d} := e^{-1} + d\gamma_{\eps,L}' \gamma_{\eps,L}^{d-1}.\]
\end{proposition}
%

The next lemma, on pointwise evaluation of the operators we are interested in, follows from Proposition \ref{p:CLTbounds}.
\begin{lemma}\label{l:Pointwise}
	For all $\phi \in C^0$, $c \leq 3\rhonm$ and $x \in \domain$,
	\[ \Pr\left(| (\D^M_\eps \phi)(x) - (\D_\eps \phi)(x)| > c \phinm\right) \leq 2\exp\left\{-\frac{M c^2}{6 \gamma_{\eps,L}^d (2\pi\eps)^{-d/2} \rhonm}\right\}. \]
\end{lemma}
\begin{proof}[Proof of Lemma \ref{l:Pointwise}]
	Equations \RR{eq:DepsAgain}{eq:DMepsAgain} and the independent sampling of the $x^i$ from $\rho$ mean that $(\D^M_\eps \phi)(x)$ is a sum of {\em i.i.d.} centred, bounded random variables:
	\[ (\D^M_\eps \phi)(x) = \frac{1}{M} \sum_{i=1}^M \mathfrak{g}_x(x^i),\]
	where
	\[ \mathfrak{g}_x(y) = g_{\eps,L}(x-y)\phi(x^i) - \E_{y}[g_{\eps,L}(x-y) \phi(y)]. \]
	The sup-norm of this function is bounded as
	\[ \|\mathfrak{g}_x\|_0 \leq 2\|g_{\eps,L}(x-\cdot)\phi(\cdot)\|_0 \leq 2 (2\pi\eps)^{-d/2} \gamma_{\eps,L}^d \phinm, \]
	and the $L^2$ norm as
	\begin{align*} \E[\mathfrak{g}_x^2] &\leq \E_y[g_{\eps,L}(x-y)^2 \phi(y)^2]\\
	&= \int g_{\eps,L}(x-y)^2 \phi(y)^2 \rho(y) \dd y\\
	&\leq (2\pi\eps)^{-d/2} \gamma_{\eps,L}^d \phinm^2 \rhonm.
	\end{align*}
	From an application of Proposition \ref{p:CLTbounds} the result then follows.
\end{proof}

By using the compactness of our domain $\domain$ we can extend this to bounds on the function norms:
\begin{lemma}\label{l:FunctionNorm}
	There exist constants ${{C}}_{41}, {{C}}_{42}$ depending only on $L,d,\rhonm,\eps_0$ such that for all $\eps < \eps_0$, $\phi \in C^0$ and $c < 3 \rhonm$,
	\[ \Pr\left(\left\| (\D^M_\eps - \D_\eps) \phi\right\|_{0} > 2c \phinm\right) \leq 2 {{C}}_{42} c^{-d} \eps^{-d(d+1)/2} \exp\left\{-{{C}}_{41} M \eps^{d/2} c^2\right\}. \]
\end{lemma}
\begin{proof}[Proof of Lemma \ref{l:FunctionNorm}]
		Firstly, we have the deterministic bound that
		\begin{equation} \Lip (\D_\eps - \D^M_\eps) \phi \leq \Lip \D_\eps \phi + \Lip \D^M_\eps \phi \leq 2\Lip g_\eps \phinm = 2\eps^{-1/2} (2\pi\eps)^{-d/2} \gamma'_{L,\eps,d} \phinm. \label{eq:LipBound}\end{equation}

		Now, define the finite subset of the domain $\domain = [0,L]^d$
		\[ S_\xi = \{ (\xi n_1,\ldots \xi n_d) : n_1, \ldots n_d = 0, \ldots, \lceil L/\xi \rceil-1. \}. \]
		No point in $\domain$ is more than $\sqrt{d} \xi$ away from an element of $S_\xi$, and $S_\xi$ contains no more than $(L/\xi+1)^d$ points.

		By applying Lemma \ref{l:Pointwise} and a union bound, we obtain that for all $x \in S_\xi$
		\[ \Pr\left(\sup_{x\in S_\xi} | (\D^M_\eps \phi)(x) - (\D_\eps \phi)(x)| > c \phinm \right) \leq 2(L/\xi+1)^d \exp\left\{-{{C}}_{41} M \eps c^2\right\}, \]
		where the constant
		\[ {{C}}_{41} := \tfrac{1}{6} \gamma_{\eps_0,L}^{-d} (2\pi)^{d/2} \rhonm^{-1}. \]

	Using the Lipschitz bound \R{eq:LipBound} we can then say that
			\begin{align*} &\Pr\left(\sup_{x\in \domain} | (\D^M_\eps \phi)(x) - (\D_\eps \phi)(x)| > (c +  2\eps^{-1/2} (2\pi\eps)^{-d/2} \sqrt{d} \gamma'_{L,\eps,d} \xi )\phinm \right) \\
			&\qquad \leq 2(L/\xi+1)^d \exp\left\{-{{C}}_{41} M \eps^{d/2} c^2\right\}. \end{align*}
		Setting
		\[\xi = \frac{c \eps^{1/2} (2\pi\eps)^{d/2} \gamma'_{L,\eps,d}}{ 2\sqrt{d}}\]
		we obtain
					\[ \Pr\left(\| (\D^M_\eps-\D_\eps) \phi\|_{0} > 2c \phinm \right) \leq 2\left( \frac{2L\sqrt{d} \eps^{-1/2} (2\pi\eps)^{-d/2}}{c\gamma'_{L,\eps,d}} + 1 \right)^d \exp\left\{-{{C}}_{41} M \eps^{d/2} c^2\right\}, \]
					which requiring that $\eps \leq \eps_0$ and setting
					\[ {{C}}_{42} = \left(2 L \sqrt{d} (2\pi)^{-d/2}/\gamma'_{L,\eps_0,d} + \eps_0^{(d+1)/2} 3 \rhonm\right)^d \]
					gives the required bound.
\end{proof}


We would now like to extend this result to convergence as operators. Recall that we defined for $\zeta > 0$ the complex domains
\[ \domain_\zeta = \{ x + iz \mid x \in \domain, z \in [-\zeta,\zeta]^d \}, \]
so that $\domain \subset \domain_\zeta \subset (\mathbb{C} / L\mathbb{Z})^d$; we also defined the Hardy spaces
\[ H^\infty(\domain_\zeta) = \{ \phi \in C^0(\domain_\zeta) : \phi \textrm{ analytic on }\intr \domain_\zeta\} \]
with $\| \cdot \|_{\zeta}$ being the $C^0(\domain_\zeta)$ norm. In Theorems \ref{t:Delta} and \ref{t:Operator} (presented in Section \ref{s:Theorems}) we show that when the size of $\zeta$ scales with the kernel bandwidth $\sqrt{\eps}$, $\D^M_\eps$ converges in operator norm to $\D_\eps$.

To extend from function-wise convergence to uniform convergence across all functions, we will again make use of a compactness argument: this time, the compact embedding of $H^\infty(\domain_\zeta)$ in $C^0(\domain)$. This choice allows us to obtain good operator convergence bounds in the strong space $H^\infty(\domain_\zeta)$ as, Gaussian convolution $\C_\eps$ maps the weak space $C^0(\domain)$ into the strong space $H^\infty(\domain_\zeta)$ with an $\bigO(1)$ penalty in norm, provided that $\zeta$ is $\bigO(\eps^{1/2})$ (see Proposition \ref{p:WeaktoStrong}).

However, this scaling restriction on $\zeta$, which arises from the width of the Gaussian kernel, leads to a complication. Because the larger complex domain $\domain_\zeta$ is only a relatively small extension of the real domain $\domain$, the number of $C^0(\domain)$ balls required for a covering of $H^\infty(\domain_\zeta)$ is exponentially large in $\eps^{-1/2}$. This jeopardises the Central Limit Theorem bounds obtained in Lemma \ref{l:FunctionNorm}.

However, we can use the Gaussian kernel's localisation to our advantage, as the values of $\D^M_\eps\phi(x), \D_\eps\phi(x)$ more or less depend only on values of $\phi$ inside a ball slightly larger than $\bigO(\eps^{1/2})$. We thus divide our domain $\domain$ up into small, overlapping cubes $\mathcal{E}$ of this size: the complex $\zeta$-fattening $\mathcal{E}_\zeta$ is a sufficiently large extension of $\mathcal{E}$ and on each of these cubes we therefore have acceptable covering numbers. 

We will make use of the following quantitative compactness result, proved in Appendix \ref{a:CompactnessProof}. \rii{Note that the analyticity of the Gaussian kernel is crucial for this result.}
\begin{proposition}\label{p:Compactness}
	Let $\mathcal{E} \subset \domain$ be a hypercube of side length $2\ell \geq 2\zeta/\eta_0 $ and, $\mathcal{E}_\zeta$ the closed $\zeta$-fattening of $\mathcal{E}$
	\[ \mathcal{E}_\zeta = \{ x \in \domain_\zeta : d(x,\mathcal{E}) \leq \zeta \}. \]

	There exist constants ${{C}}_{43}, {{C}}_{44}$ dependent only on $\eta_0, d$ such that for each $\xi\in(0,\half)$ there exists a set $\mathcal{S}^{\ell,\zeta}_\xi$ such that for every function $\phi \in H^\infty(\mathcal{E}_\zeta)$ with $\|\phi\|_{H^\infty(\mathcal{E}_\zeta)} \leq 1$,
	\begin{equation} \sup_{\psi \in \mathcal{S}^{\ell,\zeta}_\xi} \|\phi - \psi\|_{C^0(\mathcal{E})} \leq \xi, \label{eq:HardyCovering}\end{equation}
	and the cardinality of $\mathcal{S}^{\ell,\zeta}_\xi$ is bounded by
	\[ |\mathcal{S}^{\ell,\zeta}_\xi| \leq e^{({{C}}_{43} \log \xi^{-1} + {{C}}_{44} \log (\zeta^{-1}\ell))(\zeta^{-1} \ell \log \xi^{-1})^{d}}. \]
\end{proposition}

Using Proposition \ref{p:Compactness} we can prove central limit theorem-style bounds on the operator norm of $\D^M_\eps - \D_\eps$ from a strong space associated with a larger cube $\mathcal{E}$ to a weak space associated with a smaller cube $E$.
\begin{proposition}\label{p:LInfHeads}
	Let $\mathcal{E}$ be as in Proposition \ref{p:Compactness} and let $E$ be a hypercube of side length $2l < 2\ell$ centred inside $\mathcal{E}$. Then there exist positive constants ${{C}}_{44}, {{C}}_{45}, {{C}}_{46}$ dependent only on $\rhonm, d, L, \eta_0, \eps_0$ such that for $c \leq 3 \rhonm$,
	\begin{align*}&\Pr\left(\|(\D^M_\eps - \D_\eps) \ine \|_{H^\infty(\mathcal{E}_\zeta) \to C^0(E)} \geq 3c\right)\\
	&\qquad \leq \exp\left\{({{C}}_{45} \log \eps^{-1} + {{C}}_{46} \log c^{-1} + {{C}}_{44} \log (\ell/\zeta))^{d+1} (\ell/\zeta)^d - \frac{M c^2}{4 (2\pi\eps)^{-d/2} \rhonm}\right\}, \end{align*}
	where $\mathbb{1}_\mathcal{E}$ is the characteristic function of $\mathcal{E}$, considered as a multiplication operator.
	\end{proposition}
	\begin{proof}[Proof of Proposition \ref{p:LInfHeads}]
		The proof proceeds analogously to the proof of Lemma \ref{l:FunctionNorm}.

		Let $\mathcal{S}^{\ell,\zeta}_\xi$ be as in Proposition \ref{p:Compactness}. The difference between the operators can be bounded deterministically by
		 \[ \|\D^M_\eps - \D_\eps\|_0 \leq \|\D^M_\eps\|_0 + \|\D_\eps\|_0 \leq 2 \gamma_{\eps,L}^d (2\pi\eps)^{-d/2}  \]
		 and so, using \R{eq:HardyCovering},
		 \begin{equation} \sup_{\psi \in \mathcal{S}^{\ell,\zeta}_\xi} \left\| (\D^M_\eps - \D_\eps)\ine (\|\phi\|_\zeta^{-1} \phi - \psi)\right\|_{C^0(E)} \leq 2\gamma_{\eps,L}^d (2\pi\eps)^{-d/2} \xi \label{eq:HardyCoverDiff}\end{equation}
		 for all $\phi$ in the unit ball of $H^\infty(\mathcal{E}_\zeta)$.

		 On the other hand, we can apply Lemma \ref{l:FunctionNorm} and a union bound to show that
		  \begin{equation}  \Pr\left(\sup_{\psi \in \mathcal{S}^{\ell,\zeta}_\xi} \| (\D^M_\eps - \D_\eps) \ine \psi \|_{0} > 2c\right) \leq |\mathcal{S}^{\ell,\zeta}_\xi| {{C}}_{42} c^{-d} \eps^{-d(d+1)/2} \exp\left\{-{{C}}_{41} M \eps^{d/2} c^2\right\}. \label{eq:HardyCoverProb}\end{equation}

		 By combining \RR{eq:HardyCoverDiff}{eq:HardyCoverProb} and setting $\xi = (2\pi\eps)^{d/2} c/2\gamma_{\eps,L}^d$ we obtain that
		 \begin{align*}  &\Pr\left(\sup_{\|\phi\|_{H^\infty(\mathcal{E}_\zeta)} \leq 1} \| (\D^M_\eps - \D_\eps) \ine \phi \|_{0} > 3c\right) \\ &\qquad\leq |\mathcal{S}^{\ell,\zeta}_{(2\pi\eps)^{d/2} c/2\gamma_{\eps,L}^d}| {{C}}_{42} c^{-d} \eps^{-d(d+1)/2} \exp\left\{-{{C}}_{41} M \eps^{d/2} c^2\right\},
		 \end{align*}
		 which using that $\eps \leq \eps_0$, $\ell/\zeta \geq \eta_0$ and $c \leq 3 \rhonm$ and the bound on $\mathcal{S}^{\ell,\zeta}_{\xi}$ in Proposition \ref{p:Compactness} gives the required bound.
	\end{proof}

We can also make a deterministic bound on the error that this restriction to the larger cube $\mathcal{E}$ introduces relative to the full diffusion. \rii{In this proposition the Gaussian kernel's exponential decay is crucial.}
\begin{proposition}\label{p:LInfTails}
Let $\mathcal{E}, E$ be as in Proposition \ref{p:LInfHeads}. Then
	\[ \| (\D^M_\eps - \D_\eps) \notine \|_{H^\infty(\domain_\zeta) \to C^0(E)} \leq 2 (1+\gamma_{\eps,L})^d (2\pi\eps)^{-d/2} e^{-(\ell-l)^2/2\eps}. \]
\end{proposition}
\begin{proof}[Proof of Proposition \ref{p:LInfTails}]
	For $x \in E$,
	\begin{align*} |\D^M_\eps \notine \phi (x)| &\leq \frac{1}{M} \sum_{i=1}^M |g(x-x^i)| \notine(x^i) \phinm\\
	&\leq \sup_{y \in \domain\backslash\mathcal{E}} g(x-y) \phinm\\
	&\leq (2\pi\eps)^{-d/2} (1+\gamma_{\eps,L})^d e^{-(\ell-l)^2/2\eps} \phinm. \end{align*}
	Similarly,
	\begin{align*} |\D_\eps \notine \phi (x)| &\leq (2\pi\eps)^{-d/2} (1+\gamma_{\eps,L})^d e^{-(\ell-l)^2/2\eps} \phinm \end{align*}
	Combining these results and using that $\|\cdot\|_0 \leq \|\cdot\|_{\zeta}$ we obtain what is required.
\end{proof}

This is enough for us to prove Theorem \ref{t:Delta}:
\begin{proof}[Proof of Theorem \ref{t:Delta}]
	Set
	\[  \ell = 2l = \sqrt{8\eps \min\{1,\log (2(1+\gamma_{\eps,L})^d c^{-1}(2\pi\eps)^{-d/2})\}}. \]
	From Proposition \ref{p:LInfTails} we thus have
	\[ \| (\D^M_\eps - \D_\eps) \notine \|_{H^\infty(\domain_\zeta) \to C^0(E)} \leq c. \]

	Combining this with Proposition \ref{p:LInfHeads} and restricting $\ell/\zeta \leq {{C}}_{47} \log (c^{-1}\eps^{-1})$ we have
	\begin{align*} &\Pr\left(\|(\D^M_\eps - \D_\eps) \ine \|_{H^\infty(\mathcal{E}_\zeta) \to C^0(E)} \geq 4c \phinm\right)\\
	&\qquad \leq \exp\left\{({{C}}_{48} \log \eps^{-1} + {{C}}_{49} \log c^{-1})^{d+1} \log (c^{-1}\eps^{-1})^{d} - \frac{M c^2}{4 (2\pi\eps)^{-d/2} \rhonm}\right\}. \end{align*}

	The full domain $\domain$ can be covered by $\lceil L/l\rceil^d \leq (1 + L/\sqrt{8\eps})^d$ hypercubes of side-length $l$. Thus
	 \begin{equation} \Pr\left(\|(\D^M_\eps - \D_\eps) \ine \|_{H^\infty(\mathcal{E}_\zeta) \to C^0(\domain)} \geq 4c \phinm\right) \leq \exp\left\{{{C}}_{1} (\log (2\eps c)^{-1})^{2d+1} - {{C}}_{41} M \eps^{d/2} c^2\right\}, \label{eq:StrongtoWeak}
	 \end{equation}
	 which by relabelling $4c \to c$ and $\eps \to \eps/2$, and setting ${{C}}_{2} = 2^{-d/2} {{C}}_{41}/64$ gives us \R{eq:DeltaBound}, as required.
\end{proof}

The remaining necessary ingredient for the proof of Theorem \ref{t:Operator} is a bound taking one from the weak space back into the strong space. Recall the definition of the Gaussian kernel operator \R{eq:Convolution}:
\[ \mathcal{C}_\eps\phi(x) = \int_\domain g_{\eps,L}(x-y) \phi(y)\,\dd y. \]
Then the following proposition holds:
\begin{proposition}\label{p:WeaktoStrong}
	For all $\phi \in C^0(\domain)$,
	\[ \| \mathcal{C}_\eps \phi \|_{\zeta} \leq e^{d \zeta^2/2\eps} \phinm. \]
\end{proposition}
\begin{proof}[Proof of Proposition \ref{p:WeaktoStrong}]
	Extending $\phi$ periodically to $\mathbb{R}^d$, we find that
	\begin{align*}
	\| \mathcal{C}_\eps \phi \|_{\zeta} &= \sup_{x \in \domain, z \in [-\zeta,\zeta]^d} | \C_\eps \phi(x+iz) | \\
	& \sup_{x \in \domain, z \in [-\zeta,\zeta]^d} \left|\int_{\mathbb{R}^d} (2\pi\eps)^{-d/2} e^{-(x-y+iz)^2/2\eps} \phi(y)\,\dd y\right|\\
	&\leq \sup_{x \in \domain, z \in [-\zeta,\zeta]^d} \int_{\mathbb{R}^d} (2\pi\eps)^{-d/2} e^{-\Re \sum_{j=1}^d (x_j-y_j+iz_j)^2/2\eps}\phinm\,\dd y\\
	&= \sup_{z \in [-\zeta,\zeta]^d} e^{|z|^2/2\eps} \phinm,
	\end{align*}
	giving the required result.
\end{proof}

\begin{proof}[Proof of Theorem \ref{t:Operator}]
We can decompose
\[ \D^M_\eps - \D_\eps = \mathcal{C}_{\eps/2}(\D^M_{\eps/2} - \D_{\eps/2}),\]
where we recall that $\mathcal{C}_\eps$ is convolution by a Gaussian of variance $\eps$. Combining Proposition \ref{p:WeaktoStrong} and Theorem \ref{t:Delta}, we obtain the necessary bound in the $H^\infty$ norm.
\end{proof}

\section{Convergence of the weighted operator in finite data approximation}\label{s:ParticleSinkhorn}

We now turn to the normalised operator $\P^M_\eps$. We must first bound the convergence of the function $U^M_\eps$ solving the discretised Sinkhorn problem \R{eq:SinkhornDiscretisation} converges to the continuum limit $U_\eps$ solving \R{eq:SinkhornProblem}. To apply uniform bounds on $U_\eps$ and $(I-\P_\eps)^{-1}$ in the $C^0$ norm to Hardy spaces, we will use the following proposition, whose proof is in Appendix \ref{a:ResolventProof}:
\begin{proposition}\label{p:ReciprocalZeta}
	Suppose that $\phi > 0$. Then if $\zeta = Z_0 \epsilon^{1/2}$ with
	\[Z_0 \leq \frac{\pi}{8d}(\phinm \|\phi^{-1}\|_0 \rhonm \|\rho^{-1}\|_0)^{-2},\]
	then if $\psi = 1/(\D\phi)$, the bounds in the Hardy norm \R{eq:HardyNormDefinition} hold
	\begin{align*} \| \psi \|_\zeta &\leq 2 \| \phi^{-1} \|_0\\
	  \| \psi^{-1} \|_\zeta &\leq e^{2dZ_0^2} \rhonm \| \phi\|_0. \end{align*}
\end{proposition}

As an immediate consequence we have
\begin{proposition}\label{p:SinkhornWeightZeta}
	If $\zeta = Z_0 \epsilon^{1/2}$ with $Z_0 \leq \pi (\rhonm \|\rho^{-1}\|_0 {{C}}_{22}^2)^{-2}/8d$, where ${{C}}_{22}$ is defined in Theorem \ref{t:UConvergence}, then
	\begin{equation} \|U_\eps\|_\zeta \leq 2 {{C}}_{22}^{-1}, \label{eq:WeightZetaNorm}\end{equation}
\end{proposition}

We will also find the following proposition useful:
\begin{proposition}\label{p:SinkhornResolventZeta}
	There exists a constant ${{C}}_{51}$ such that for all $\eps < \eps_0$  and $Z_0$ as in Proposition \ref{p:SinkhornWeightZeta}, then
	\begin{equation} \|(I+\P_\eps)^{-1}\|_\zeta \leq {{C}}_{51}.
	\label{eq:ResolventZetaNorm}\end{equation}
\end{proposition}

To prove this proposition we require the following result, whose proof is in Appendix \ref{a:ResolventProof}.
\begin{lemma}\label{l:Resolvent}
	There exists a constant ${{C}}_{52}$ such that for all $\eps \leq \eps_0$,
	\[ \|(I + \P_\eps)^{-1}\|_0 \leq {{C}}_{52}. \]
\end{lemma}

\begin{proof}[Proof of Proposition \ref{p:SinkhornResolventZeta}]
	We decompose
	\[ (I - \P_\eps)^{-1} = I + \P_\eps (I - \P_\eps)^{-1}. \]
	We then have for $\phi \in H^\infty(\domain_\zeta)$ that
	\[ \| \P_\eps (I - \P_\eps)^{-1} \phi \|_\zeta \leq \|U_\eps\|_\zeta \|D_\eps\|_{0 \to \zeta} \Uepsnm \|(I - \P_\eps)^{-1}\|_0 \|\phi\|_0, \]
	which by an application of Proposition \ref{p:SinkhornWeightZeta} and Lemma \ref{l:Resolvent} gives
	\[ \| \P_\eps (I - \P_\eps)^{-1} \phi \|_\zeta \leq 2{{C}}_{22}^2 e^{2dZ_0^2} {{C}}_{52} \|\phi\|_0. \]
	Using that $\|\cdot\|_0 \leq \|\cdot\|_\zeta$ we obtain the required result.
\end{proof}

We can now prove convergence of the Sinkhorn weight as the number of particles $M \to \infty$:
\begin{lemma}\label{l:Sinkhorn}
	Suppose $Z_0$ is as in Proposition \ref{p:SinkhornWeightZeta}. There exist constants ${{C}}_{9}, {{C}}_{53}$ 
	such that if $\delta \leq {{C}}_{9}$ then
	\[ \| U^M_\eps - U_\eps\|_\zeta,\, \| Y^M_\eps - Y_\eps\|_0,\, \| (Y^M_\eps)^{-1} - (Y_\eps)^{-1}\|_0 \leq {{C}}_{53} \delta, \]
	where $\zeta = Z_0 \eps^{1/2}$.
\end{lemma}

\begin{proof}[Proof of Lemma \ref{l:Sinkhorn}]
	We can rewrite \R{eq:SinkhornProblem} and \R{eq:SinkhornDiscretisation} as
	\begin{align*}
	U_\eps(x)\, (\D_{\eps}U_\eps)(x) &\equiv 1\\
	U^M_\eps(x) \, (\D^M_{\eps}U^M_\eps)(x) &\equiv 1.
	\end{align*}
	If for $\theta \in [0,1]$ we set
	\[ \D^\theta_\eps := (1-\theta)\D_{\eps} + \theta\D^M_{\eps} \]
	then we obtain a one-parameter family of Sinkhorn weight functions $U^\theta_\eps$ solving
	\begin{equation} U^\theta_\eps(x)\, (\D^\theta_{\eps}U^\theta_\eps)(x) \equiv 1. \label{eq:SinkhornContinuum}\end{equation}
	The existence and uniqueness of the $U^\theta_{\eps}$ follow from the positivity of the operator $\D^\theta_{\eps}$, on $L^\infty(\domain)$ for $\theta\in[0,1)$ and on $L^\infty(\{x^i\}_{i=1\ldots M})$ for $\theta = 1$.

	Furthermore, because
	\[\frac{\dd}{\dd\theta} \D^\theta_\eps = \D^M_\eps - \D_{\eps} \]
	is a bounded operator on $H^\infty_\zeta$, we can apply the implicit function theorem to \R{eq:SinkhornContinuum} as long as $U^\theta_\eps$ stays in $H^\infty_\zeta$, so that
	\[ \frac{\dd}{\dd\theta} \log U^\theta_\eps = -(I + U^\theta_\eps \D^\theta_{\eps}U^\theta_\eps)^{-1} U^\theta_\eps(x) \left((\D^M_\eps - \D_{\eps})U^\theta_\eps\right)(x). \]

	We have from Propositions \ref{p:SinkhornWeightZeta} and \ref{p:SinkhornResolventZeta} that $\|U_\eps\|_\zeta \leq 2{{C}}_{22}$, and $\|(I-\P_\eps)^{-1}\|_\zeta \leq {{C}}_{51}$, and from Theorem \ref{t:Operator} that $\|\D^M_\eps - \D_\eps\|_\zeta \leq e^{2d Z_0^2} \delta$. Note that since $\P_\eps$ has $1$ as an eigenvalue, ${{C}}_{51} \geq 1/2$.

	If $ B(\theta) := \| \log U^\theta_{\eps} - \log U_\eps \|_\zeta$, then
	\[ B'(\theta) \leq \left\| \frac{\dd}{\dd\theta} \log U^\theta_\eps\right\|_\zeta
	\leq {{C}}_{51} (1 - {{C}}_{51} \|U^\theta_\eps \D^\theta_{\eps}U^\theta_\eps - \P_\eps \|_\zeta)^{-1} e^{2d Z_0^2} \delta \|U^\theta_\eps\|_\zeta.
	\]
	Because $\|U^\theta_\eps\|_\zeta \leq 2{{C}}_{22} e^{B(\theta)}$,
	\[ B'(\theta)
	\leq \frac{4{{C}}_{22}^2 {{C}}_{51} e^{2d Z_0^2} \delta e^{2B(\theta)}}{1 - 4{{C}}_{22}^2 {{C}}_{51} (\|\D^\theta_{\eps}\|_\zeta B(\theta)(e^{2B(\theta)}+e^{B(\theta)} + e^{2d Z_0^2} \delta e^{B(\theta)}) )}
	\]
	and because $\|\D^\theta_{\eps}\|_\zeta \leq \|\D_{\eps}\|_\zeta + \theta e^{2d Z_0^2} \delta \leq 1 + e^{2d Z_0^2} \delta$,
	\[ B'(\theta)
	\leq \frac{4{{C}}_{22}^2 {{C}}_{51} e^{2d Z_0^2} \delta e^{2B(\theta)}}{1 - 4{{C}}_{22}^2 {{C}}_{51} (2(1+e^{2d Z_0^2} \delta)B(\theta)e^{2B(\theta)} + e^{2d Z_0^2} \delta e^{B(\theta)}) )} .
	\]

	Thus, as long as $B(\theta) \leq \min\{e^{2d Z_0^2} \delta, {{C}}_{54}\}$ and $\delta \leq {{C}}_{55}/(4{{C}}_{22}^2 {{C}}_{51} e^{2d Z_0^2}) =: {{C}}_{59}$ for some fixed constants ${{C}}_{54},{{C}}_{55},{{C}}_{56}$,
	\[ B'(\theta) \leq {{C}}_{56} 4{{C}}_{22}^2 {{C}}_{51} e^{2d Z_0^2} \delta \]
	and thus
	\[  \|\log U_\eps - \log U^M_\eps\|_\zeta \leq B(1) \leq {{C}}_{56} 4{{C}}_{22}^2 {{C}}_{51} e^{2d Z_0^2} \delta. \]

	Furthermore, for some fixed constant ${{C}}_{57}$,
	\begin{align*}\|U^M_\eps - U_\eps \|_\zeta &\leq \|U_\eps\|_\zeta \left(e^{\|\log U_\eps - \log U^M_\eps\|_\zeta} - 1\right)\\
	& \leq {{C}}_{57} 8{{C}}_{22}^3 {{C}}_{51} e^{2d Z_0^2} \delta\\
	& \leq {{C}}_{57} 8 {{C}}_{22}^3 {{C}}_{51} e^{2 d Z_0^2} \delta =: {{C}}_{58} \delta,
	\end{align*}
	as required.

	To prove the second part, we use the definition of $Y^{(M)}_\eps$ in \R{eq:Y} to say that
	\[ Y^M_\eps - Y_\eps = (\D^M_{\eps/2} - \D_{\eps/2}) U^M_\eps - \D_{\eps/2} (U^M_\eps - U_\eps) \]
	and so
	\[ \| Y^M_\eps - Y_\eps \|_0 \leq (2{{C}}_{22} + {{C}}_{59} + {{C}}_{58})\delta =: {{C}}_{60} \delta. \]
	Furthermore,
	\[ (Y^M_\eps)^{-1} - (Y_\eps)^{-1} = \frac{Y_\eps^{-2} (Y^M_\eps - Y_\eps)}{1-Y_\eps^{-1} (Y_\eps - Y^M_\eps)} \]
	and so using that
	\[ \|Y_\eps^{-1}\|_0 \leq \|\rho^{-1}\|_0 \|U_\eps^{-1}\|_0 \leq \|\rho^{-1}\|_0 \|\rho U_\eps\|_0 = \|\rho^{-1}\|_0 \rhonm {{C}}_{22}, \]
	we have that provided that $ \delta < \min\{{{C}}_{59},({{C}}_{60} \half \|\rho^{-1}\|_0 \rhonm {{C}}_{22})^{-1}\} =: {{C}}_{9}$,
	\[ \|(Y^M_\eps)^{-1} - (Y_\eps)^{-1}\|_0 \leq 2(\|\rho^{-1}\|_0 \rhonm {{C}}_{22})^2 {{C}}_{60}\delta =: {{C}}_{53}\delta. \]
	Readjusting ${{C}}_{53} = \max\{{{C}}_{58},{{C}}_{60},{{C}}_{53}\}$, we have what is required.
\end{proof}

The convergence of the Sinkhorn-weighted operator then follows in Theorem \ref{t:WeightedOperatorConvergence}, which we prove here:
\begin{proof}[Proof of Theorem \ref{t:WeightedOperatorConvergence}]
	We can decompose
	\[ \P^M_\eps - \P_\eps = U^M_\eps (\D^M_\eps - \D_{\eps})U^M_\eps + U^M_\eps\D_\eps (U^M_\eps - U_\eps) + (U^M_\eps - U_\eps)\D_\eps U_\eps.\]
	Using Lemma \ref{l:Sinkhorn}, Propositions \ref{p:SinkhornWeightZeta} and \ref{p:SinkhornResolventZeta} and that $\|\D_{\eps}\|_\zeta \leq \rhonm$ we have that
	\begin{align*} \| \P^M_\eps - \P_\eps \|_\zeta &\leq \big((2{{C}}_{22} + {{C}}_{53} {{C}}_{9})^2 {{C}}_{53}  + (2{{C}}_{22} + {{C}}_{53} {{C}}_{9}) \rhonm {{C}}_{53}  + \\
	&\qquad {{C}}_{53}  \rhonm 2{{C}}_{22} \big) e^{2 dZ_0^2} \delta, \\
	& =: {{C}}_{10}\delta \end{align*}
	for some constant ${{C}}_{10}$.

	The corresponding bounds for the half-step operators $\J^M_\eps,\, \K^M_\eps$ and the semi-conjugate operator $\M^M_{\eps,1}$ arise similarly, with an appropriate adjustment of ${{C}}_{10}$; this extends to general $\M^M_{\eps,n} = (\M^M_{\eps,1})^n$ by using that $\M^M_{\eps,1}, \M_{\eps,n}$ are row-stochastic and thus have unit $C^0$ norm.
\end{proof}

\section{Convergence of spectral data}\label{s:Spectral}

We can now combine our ``bias'' and ``variance'' operator errors to obtain the convergence of the spectral data. However, instead of studying the perturbed operators $\P^{(M)}_\eps = \J^{(M)}_\eps \K^{(M)}_\eps$, we will consider semi-conjugacies $\M^{(M)}_{\eps,1} = \K^{(M)}_\eps\J^{(M)}_\eps$, so that we can use the function space $C^0$ consistently across the two limits. The outline of our attack is standard \citep{Keller99}: we will first establish the convergence of resolvents in a strong space-to-weak space operator norm sense, and then use this to bound the error in the discretised operators' spectrum, and in spectral projection operators (and thus eigenspaces).

While the variance error $\M^M_{\eps,n} - \M_{\eps,n}$ is just a perturbation in operator norm (from Theorem \ref{t:WeightedOperatorConvergence}), the bias error $\M_{\eps,n} - e^{n\eps \calL}$ is only small from the strong space $C^{3+\beta}$ into the weak space $C^0$. To obtain convergence of resolvents we must therefore quantify the regularising behaviour of the operators $\M_{\eps,n}$ from the weak space into the strong space:
\begin{proposition}\label{p:KJepsContraction}
	Suppose $\rho \in \W{s},\, s > 2$, and $\beta \in (0,\min\{s-2,1\})$. For all $\tilde T > 0$ there exists a constant ${{C}}_{61,\beta}$ depending on $\tilde T, \rho, \beta$ such that for all $\eps\leq\eps_0$, $n\eps \geq \tilde T$,
	\[ \| \M_{\eps,n} \|_{C^0 \to C^{3+\beta}} \leq {{C}}_{61,\beta}, \]
	where $\M_{\eps,n}$ is defined in \R{eq:M}.
\end{proposition}
\begin{proof}
	For $\tilde T \leq n \eps \leq \tilde T + \eps_0$ this is a Schauder estimate \citep{Knerr80}, which can be extended from $C^{2+\beta}$ to $C^{3+\beta}$ along the lines of Proposition \ref{p:SchauderC3}. For larger $n$, this follows by using that $\|\M_{\eps,1}\|_{C^0} = 1$.
\end{proof}

Let us denote the resolvent of an operator $\mathcal{A}$ as
\begin{equation} R_\lambda(\mathcal{A}) := (\lambda I - \mathcal{A})^{-1}. \label{eq:ResolventDef} \end{equation}

We have the following bound on the resolvent of our semigroup in the $C^0$ norm:
\begin{proposition}\label{p:SemigroupResolvent}
	For all $\tilde T > 0$ there exists a constant ${{C}}_{62}$ depending on $\tilde T, \rho$ such that for $t > \tilde T$
	\[ \| R_\lambda(e^{t \calL})\|_{C^0} \leq |\lambda|^{-1}(1 + {{C}}_{62} d(\lambda, e^{t \sigma(\calL)})^{\rii{-1}}), \]
	where the point-to-set distance $d(\lambda, A) = \inf_{\mu \in A} d(\lambda,\mu)$.
\end{proposition}
\begin{proof}
		Using that $\| \cdot \|_{C^0} \geq \|\cdot\|_{L^2(\rho)}$ we have
	\[  \| R_\lambda(e^{t \calL}) \|_{C^0} \leq |\lambda|^{-1}\left(1 + \|e^{t \calL}\|_{L^2(\rho) \to C^0} \|R_\lambda(e^{t \calL})\|_{L^2(\rho)}\right).\]
	Upper Gaussian estimates on $e^{\tilde T \calL}$ \citep{Liskevich00} mean that we can bound
	\[  \|e^{t \calL}\|_{L^2(\rho) \to C^0} \leq \|e^{\tilde T \calL}\|_{L^2(\rho) \to C^0} \|e^{(t-\tilde T) \calL}\|_{L^2(\rho)} = \|e^{\tilde T \calL}\|_{L^2(\rho) \to C^0} =: {{C}}_{62}. \]
	Furthermore, since $e^{t\calL}$ is normal in $L^2(\rho)$ with spectrum $\sigma(e^{t\calL}) = e^{t\sigma(\calL)}$, the resolvent's norm is bounded by \rii{the reciprocal of} the distance to the spectrum
	\[ \|R_\lambda(e^{t \calL})\|_{L^2(\rho)} = d(\lambda, e^{t \sigma(\calL)})^{\rii{-1}}, \]
	giving us what is required.
\end{proof}

The following result then allows us to extend the previous bound to resolvents of the discretised operators:
\begin{lemma}\label{l:ResolventError}
	Suppose $n\eps \in [\tilde T, T]$, and let the quantities
		\begin{align*} X_1 &= |\lambda|^{-2} (1 + {{C}}_{62} d(\lambda, e^{t \sigma(\calL)})^{\rii{-1}}) {{C}}_{7,T},\\
	X_2 &= 1 + |\lambda|^{-1} (1 + {{C}}_{62} d(\lambda, e^{t \sigma(\calL)})^{\rii{-1}}),\\
	 X_3 &=  |\lambda|^{-1} {{C}}_{10} T \eps^{-1} \delta X_2.
	\end{align*}
	 Then if $\delta < {{C}}_{10}$,
	\begin{enumerate}[(a)]
		\item If ${{C}}_{61,\beta} X_1 \leq 1$, then $\Rl(\Me)$ is bounded in $C^0$ and
 		\[ \|\Rl(\Me) - \Rl(\eL)\|_{C^{3+\beta} \to C^0} \leq \frac{X_2}{1 - {{C}}_{61,\beta}X_1} X_1. \]
		\item If ${{C}}_{61,\beta} X_1 + X_3 \leq 1$, then $\Rl(\Me^M)$ is bounded in $C^0$ and
		\[ \| \Rl(\Me^M) - \Rl(\Me)\|_{C^0} \leq \frac{|\lambda|^{-1} \tilde X_1 X_3}{(1 - {{C}}_{61,\beta} X_1 - X_3)(1 - {{C}}_{61,\beta} X_1)}.\]
	\end{enumerate}
\end{lemma}
\begin{proof}[Proof of Lemma \ref{l:ResolventError}]
	By algebraic manipulations we have both that
	\begin{equation} \Rl(\Me) = \lambda^{-1}(I + \Me \Rl(\Me)) \label{eq:Resolvent1} \end{equation}
	and
	\begin{equation} \Rl(\Me) = \Rl(\eL) + \lambda \mathcal{X}_1 \Rl(\Me) , \label{eq:Resolvent2} \end{equation}
	where the operator
	\[\mathcal{X}_1 = \lambda^{-1} \Rl(\eL) (\Me-\eL).\]
	By substituting \R{eq:Resolvent2} into \R{eq:Resolvent1}, we then have that
	\[ (I -  \mathcal{X}_1 \Me) \Rl(\Me) = \lambda^{-1} \mathcal{X}_2, \]
	where
	\[ \mathcal{X}_2 := I + \Me \Rl(\eL). \]
	We then have using Theorem \ref{t:DeterministicOperator} that
	\begin{align*} \|\mathcal{X}_1\|_{C^{3+\beta}\to C^0} &\leq |\lambda|^{-1}  \|\Rl(\eL)\|_{C^0} \|\Me-\eL\|_{C^{3+\beta} \to C^0} \leq X_1,\\
	 \|\mathcal{X}_2\|_{C^0} &\leq 1 +  \|\Rl(\eL)\|_{C^0} \|\Me\|_{C^0} \leq X_2,
	\end{align*}
	so, using Proposition \ref{p:KJepsContraction}, if ${{C}}_{61,\beta} X_1 < 1$ then
	\begin{equation} \| \Rl(\Me) \|_{C^0} \leq \frac{|\lambda|^{-1}X_2}{1 - {{C}}_{61,\beta} X_1} < \infty. \label{eq:ResolventEpsBound} \end{equation}

  	Then, substituting \R{eq:Resolvent1} instead into \R{eq:Resolvent2} and rearranging, we obtain that
 		\[ (I - \mathcal{X}_1 \Me)(\Rl(\Me) - \Rl(\eL)) = \mathcal{X}_1 \mathcal{X}_2 \]
	so that again if ${{C}}_{61,\beta} X_1 < 1$,
 		\[ \|\Rl(\Me) - \Rl(\eL)\|_{C^{3+\beta} \to C^0} \leq \frac{X_2}{1 - {{C}}_{61,\beta}X_1} X_1, \]
	as required for (a).
\\

For part (b), we have that
\[ (\Rl(\Me^M) - \Rl(\Me))(I - (\Me^M - \Me) \Rl(\Me)) = \Rl(\Me) (\Me^M - \Me) \Rl(\Me). \]
We also have from Theorem \ref{t:WeightedOperatorConvergence} that
\[ \| \Me^M - \Me \|_{C^0} \leq {{C}}_{10} n \delta \leq {{C}}_{10} T \eps^{-1} \delta, \]
and so, using \R{eq:ResolventEpsBound},
\[ \|(\Me^M - \Me) \Rl(\Me)\|_{C^0} \leq \frac{X_3}{1 - {{C}}_{61,\beta} X_1}. \]
Consequently if
\[ X_3 + {{C}}_{61,\beta} X_1 < 1,\]
then
\[ \| \Rl(\Me^M) - \Rl(\Me)\|_{C^0} \leq \frac{|\lambda|^{-1} X_2 X_3}{(1 - {{C}}_{61,\beta} X_1 - X_3)(1 - {{C}}_{61,\beta} X_1)}  \]
and in particular, $\Rl(\Me^M)$ is bounded.
\end{proof}

Using Lemma \ref{l:ResolventError}, we can prove Theorem \ref{t:Eigendata}. \rii{This uses ideas from the Keller-Liverani spectral stability theorem \citep{Keller99}, in particular restricting the spectrum of the perturbed operator to areas where we cannot show that it is bounded, and studying spectral projection operators that we construct from resolvents and which allow us to bound convergence of the eigenspaces.}
\begin{proof}[Proof of Theorem \ref{t:Eigendata}]
	Fix $T$ and set $\tilde T = T + 2\eps_0$ and $n = \lceil\tilde T/\eps\rceil$.

	The eigenvalues of $\Me^{(M)}$ are $e^{n \eps \lambda_{k,\eps}^{(M)}}$, and the eigenvalues of $\eL$ are $e^{n \lambda_k}$. On $[e^{-T\lambda_*},1]$, the logarithm function is bi-Lipschitz, so bounds on the errors in the eigenvalues of $\Me^{(M)}$ translate to the bounds necessary for the theorem.

	 By considering the constants in Lemma \ref{l:ResolventError} we find that for $\eps^2$ and $\eps^{-1} \delta$ sufficiently small, the resolvents of \rii{the iterated perturbed operators} $\Rl(\Me),\Rl(\Me^M)$ are bounded respectively for
	\begin{align*} d(\lambda, \sigma(\eL)) &> {{C}}_{5} \eps^2, \\
	 d(\lambda, \sigma(\eL)) &> {{C}}_{6} (\eps^2 + \eps^{-1} \delta).
	\end{align*}
	\rii{This restricts the spectrum of $\Me^{(M)}$ to a small neighbourhood of the spectrum of $\eL$. Furthermore, if ${{C}}_{6} (\eps^2 + \rii{\eps^{-1}} \delta)$ and ${{C}}_{5} \eps^2$ are both smaller than $r_{\lambda_*}$, which we define to be one-quarter the smallest gap between any elements of $\{e^{-\tilde T\lambda_*}\} \cup (\sigma(e^{-\tilde T\calL})\cap [e^{-\tilde T\lambda_*},1])$, then this neighbourhood can be decomposed into disjoint open balls centered on the elements of $\lambda_k \in \sigma(\eL)$. Furthermore, the multiplicity of the spectrum of $\Me^{(M)}$ in each of these balls must be the same as the multiplicity of $\lambda_k$ as an eigenvalue of $\sigma(\eL)$. This proves part (a).}
	\\

	For part (b), we will use the bounds on the error of the resolvents in Lemma \ref{l:ResolventError} to obtain bounds on the eigenspaces using spectral calculus. Let use define the spectral projections $\Pi_k, \Pi_{k,\eps}^{(M)}$, where $\Pi_k$ is the $L^2(\rho)$-orthogonal projection onto $E_k$ satisfying
	\begin{equation} \Pi_k := \frac{r}{2\pi} \int_{0}^{2\pi} R_{e^{-nk\lambda_k} + r e^{i\theta}}(\eL)\, \dd \theta. \label{eq:SemigroupProjectionDefinition} \end{equation}
	 and 
	\begin{equation} \Pi_{k,\eps}^{(M)} := \frac{r}{2\pi} \int_{0}^{2\pi} R_{e^{-nk\lambda_k} + r e^{i\theta}}(\Me^{(M)})\, \dd \theta. \label{eq:ProjectionDefinition} \end{equation}
	\rii{Suppose the condition for disjoint balls described in (a) is satisfied around $\lambda_k$, i.e. that ${{C}}_{5} (\eps^2 + \eps^{-1} \delta), {{C}}_6 \eps < r_{\lambda_*}$, and $\lambda_k > \lambda_*$}. Then the $\Pi_{k,\eps}^{(M)}$ are projections onto the finite-dimensional spaces $\K^{(M)}_\eps \bar E_{k,\eps}^{(M)}$.

	Choose $\beta \in (0,\min\{s-2,1\})$, where $s$ is such that $\rho \in \W{s}$. Lemma \ref{l:ResolventError} and \RR{eq:SemigroupProjectionDefinition}{eq:ProjectionDefinition} give us that
	\[ \| \Pi_{k,\eps}^{M} - \Pi_{k} \|_{C^{3+\beta} \to C^0} \leq {{C}}_{63} (\eps^2 + \eps^{-1} \delta). \]
	If $\phi \in E_{k}$ with $\|\phi\|_{C^0} = 1$, then
	\begin{align*} \| e^{\eps \calL/2} \Pi_{k} \phi - \J^M_\eps \Pi_{k,\eps}^M \phi \|_{C^0} &\leq \| (e^{\eps \calL/2} - \J^M_\eps) \Pi_{k} \phi - \J^M_\eps (\Pi_{k} - \Pi_{k,\eps}^M) \phi \|_{C^0}\\
	& \leq \left((\| e^{\eps \calL/2} - \J_\eps \|_{C^{3+\beta} \to C^0}+\|  \J_\eps - \J^M_\eps \|_{C^0}) \| \Pi_{k}\|_{C^{3+\beta}} \right.\\
	& \qquad + {{C}}_{63} (\eps^2 + \eps^{-1} \delta)\Big) \| \phi \|_{C^{3+\beta}}\\
	& \leq \left({{C}}_{7,\tilde T} \eps^2 +{{C}}_{10}\delta) \| \Pi_{k}\|_{C^{3+\beta}} + {{C}}_{63} (\eps^2 + \eps^{-1} \delta)\right) \| \phi \|_{C^{3+\beta}},
	\end{align*}
	where in the last line we used Theorems \ref{t:DeterministicOperator} and \ref{t:WeightedOperatorConvergence}.

	We know that $\Pi_{k}$ is bounded on $C^{3+\beta}$ independent of $M$ and $\eps$; we also know that $E_{k}$ is a finite-dimensional subspace of $C^{3+\beta}$ and thus the $C^0$ and $C^{3+\beta}$ norms are equivalent. (The relevant constants are bounded by the usual Schauder and Gaussian estimates on $\mathcal{L}$ since $\Pi_{k}$ is a contraction on $L^2(\rho)$ and the growth of elements of $E_{k}$ under $e^{t\mathcal{L}}$ is controlled.) As a result,\rii{for $\phi \in C^0$ we have}
	\begin{align*}
	 \| e^{\eps \calL/2} \Pi_{k} \phi - \J^M_\eps \Pi_{k,\eps}^M \phi \|_{C^0} &\leq {{C}}_{64}(\eps^2 + \eps^{-1} \delta) \| \phi \|_{C^0}\\
	&\leq e^{\lambda_*\eps/2}{{C}}_{64}(\eps^2 + \eps^{-1} \delta) e^{-\lambda_k \eps/2} \| \phi \|_{C^0}, \end{align*}

	Now, $e^{\eps \calL/2} \Pi_{k} \phi = e^{-\lambda_k \eps/2} \phi$, and $\J^M_\eps \Pi_{k,\eps}^M \phi \in \bar E_{k,\eps}^M$, so
	\[ d(\phi, \bar E_{k,\eps}^M) \leq {{C}}_{65}(\eps^2 + \eps^{-1} \delta). \]
	As a result of Lemma 1 in \citet{Osborn75}, we have what we need for $\bar E_{k,\eps}^M$; the equivalent for $\bar E_{k,\eps}$ holds similarly.
\end{proof}

\begin{proof}[Proof of Corollary \ref{c:GraphLaplacian}]
	The difference between the graph Laplacian eigenvalues and the semigroup eigenvalues can be bounded
	\begin{align*} | \tilde \lambda_{k,\eps}^{(M)} - \lambda_{k,\eps}^{(M)} | &= \eps^{-1}| -e^{-\eps \lambda_{k,\eps}^{(M)}} + 1 - \eps \lambda_{k,\eps}^{(M)}|\\
	&\leq \half \eps^{-1} (\eps \lambda_{k,\eps}^{(M)})^2 \\
	&\leq \half \lambda_*^2 \eps \leq {{C}}_{63} \eps,
	\end{align*}
	where in the second-last inequality we used from the proof of Theorem \ref{t:Eigendata} that $-\lambda_{k,\eps}^{(M)}$ is forced to be greater than $-\lambda_*$. Combining this bound with Theorem \ref{t:Eigendata}(a) gives part (b); part (a) follows similarly from Theorem \ref{t:EigendataStandard}(a).
\end{proof}

\section{Results for standard weights}\label{s:StandardWeights}

In this section we will sketch the proof of Theorem \ref{t:EigendataStandard} on the convergence of spectral data for standard weights. For the most part this closely follows the argument for the Sinkhorn weights, however it is somewhat simpler in that the weights are explicitly given, and the bias error is only first-order in the timestep so the averaging argument is not necessary.
\\

We will again study the bias error by interpolating $\check\P_{\eps,\alpha}$ in time as a PDE.  We begin by bounding the associated drift term: as with the Sinkhorn weights, we will need to venture into bounding the norm of inverse derivatives of the drift terms.
\begin{proposition}\label{p:RhoEpsError}
	Suppose $\rho \in \W{s}, s > \tfrac{3}{2}$ and let $r^* = \max\{0,2-s\}$. Then there exists a constant ${{C}}_{71,s}$ such that for all $\eps \geq 0$,
	\[ \| J^{-s_*/2} \log \rho_\eps - \log \rho \|_{\W{s-2+r^*}} \leq {{C}}_{71,s} \eps. \]
\end{proposition}
\begin{proof}
	Because $J$ and $\Delta$ commute and $\|e^{t\Delta}\|_{L^\infty} \leq 1$, for all $t$
	\[ \| \rho_t \|_{\W{s}} = \| e^{t\Delta/2} \rho \|_{\W{s}} \leq \| \rho \|_{\W{s}}. \]
	Consequently if we set $\omega^t := J^{-s_*/2} \log \rho_t$, then because $\rho_t \geq \inf \rho$,
	\[ \| \omega^t \|_{\W{s+r^*}} = \| \log \rho_t \|_{\W{s}} \leq {{C}}_{72,s} \]
	 for some ${{C}}_{72,s}$.
	Then, because
	\begin{align*} \partial_t \omega^t &= \half \Delta \omega^t + \half J^{-s_*/2} |\nabla J^{r^*/2} \omega^t|^2,
	\end{align*}
	there exists ${{C}}_{71,s}$ such that for all $t \in [0,\eps_0]$,
	\[ \| \partial_t \omega^t \|_{\W{s-2+r^*}} \leq {{C}}_{71,s}, \]
	and so
	\[ \| J^{-s_*/2} (\log \rho_\eps - \log \rho) \|_{\W{s-2+r^*}} = \| \omega_\eps - \omega_0 \|_{\W{s-2+r^*}} \leq {{C}}_{71,s} \eps, \]
	as required.
\end{proof}

\begin{proposition}\label{p:UErrorStandard}
	Suppose $\rho \in \W{s}, s > \tfrac{3}{2}$. Let $\check w_{\eps,\alpha}^t = \log (\D_{\eps\{t/\eps\}}\check U_{\eps,\alpha}) - (1-\alpha) \log \rho$, and let $r^* = \max\{0,2-s\}$. Then there exists a constant ${{C}}_{73,s}$ such that for all $\eps \in [0,\eps_0]$,
	\[ \| J^{-s_*/2}\check w_{\eps,\alpha}^t \|_{\W{s-2+r^*}} \leq {{C}}_{73,s} \eps. \]
\end{proposition}
\begin{proof}
	Because $\rho_\eps \leq \inf \rho$, we know that for all $\eps \in [0,\eps_0]$,
	\[ \| \rho \check U_{\eps,\alpha} \|_{\W{s}} = \| \rho \rho_\eps^{-\alpha} \|_{\W{s}} \leq {{C}}_{74,s} \]
	for some constant ${{C}}_{74,s}$. As in Proposition \ref{p:RhoEpsError} this gives us uniform boundedness of $J^{-s_*/2} \check w_{\eps,\alpha}^t$ in $\W{s+r^*}$, and so by a similar argument we obtain the required result.
\end{proof}

\rii{
At this point we recall that the following function and operators were defined in Section \ref{ss:Operators} implicitly, by analogy with the Sinkhorn-weighted operators:
\begin{align}
Y^{(M)}_\eps(x) &= (\D^{(M)}_{\eps/2} \check U^{(M)}_{\eps.\alpha})(x).  \label{eq:YStandard}\\
\check \J^{(M)}_{\eps,\alpha} &= \check V^{(M)}_{\eps,\alpha} \C_{\eps/2} \check Y^{(M)}_{\eps,\alpha} \label{eq:JStandard} \\
\check \K^{(M)}_{\eps,\alpha} &= (\check Y^{(M)}_{\eps,\alpha})^{-1} \D^{(M)}_{\eps/2} \check U^{(M)}_{\eps,\alpha}. \label{eq:KStandard}\\
\check \M^{(M)}_{n,\eps,\alpha} &= (\check \K^{(M)}_{\eps,\alpha} \check \J^{(M)}_{\eps,\alpha})^n \label{eq:MStandard}.
\end{align}
}

The following proposition bounds the convergence of the continuum operator $\check\P_{\eps,\alpha}$ as $\eps \to 0$. In this case an averaging result is not necessary: we only need to bound the drift term.
\begin{proposition}\label{p:DeterministicOperatorStandard}
	Suppose $\rho \in \W{s},\, s > \tfrac{3}{2}$, and let $\check S_{\eps,\alpha}(t_1,t_0)$ be the solution operator of the PDE
	\begin{equation} \partial_t \phi^t = \check \calL_\alpha \phi^t + \nabla \check w_{t}^{\eps,\alpha} \cdot \nabla \phi^t, \label{eq:SDMPDEStandard}\end{equation}
	where $\check w_t^{\eps,\alpha} := \log (\D_{\eps\{t/\eps\}}\rho_\eps^{-\alpha}) - (1-\alpha) \log \rho$. 

	Then
	\begin{align*}
	\check\J_{\eps,\alpha} &= \check S_{\eps,\alpha}(\eps,\half\eps)\\
	\check\K_{\eps,\alpha} &= \check S_{\eps,\alpha}(\half\eps,0)\\
	\check\P_{\eps,\alpha} &= \check S_{\eps,\alpha}(\eps,0)\\
	\check\M_{\eps,n,\alpha} &= \check S_{\eps,\alpha}((n+\half)\eps,\half\eps).
	\end{align*}
	Furthermore, for all $T>0, \beta \in (\half, \min\{1,s-1\})$ there exists a constant $\check {{C}}_{7,T,\beta,\alpha}$ such that for all $0 \leq t_1-t_0 \leq T$ and $\eps \leq \eps_0$,
	\begin{align} \|\check S_{\eps,\alpha}(t_1,t_0) - e^{(t_1-t_0) \check\calL_{\alpha}}\|_{C^{2+\beta}\to C^0} \leq \check {{C}}_{7,T,\beta,\alpha} \eps. \label{eq:GeneratorBoundStandard} \end{align}
\end{proposition}
\begin{proof}
	The first part is as in the proof of Theorem \ref{t:DeterministicOperator}.

	From Proposition \ref{p:UErrorStandard} we have that for $r \in (\tfrac{3}{2},\min\{2,s\})$,
	\[ \| J^{-r_*/2-1} \nabla \check w_{\eps,\alpha}^t \|_{\W{s-2+r^*}} \leq K^\nabla {{C}}_{73,r}\eps, \]
	where $r^* = \max\{0,2-r\}$.

	We then have that for $\beta \in (\half,\min\{1,r-1\})$,
	\begin{align} &\| \check S_{\eps,\alpha}(t_1,t_0) - e^{(t_1-t_0) \check\calL_{\alpha}}\|_{C^{2+\beta} \to C^0} \notag\\
		&\qquad \leq \int_{t_0}^{t_1} \|e^{(\tau - t_0) \check\calL_{\alpha}} \nabla \check w_{t}^\eps \cdot \nabla\|_{\W{r} \to C^0} \| S_\eps(t_1,\tau) \|_{C^{2+\beta} \to \W{r}}\,\dd \tau.
		\label{eq:SolutionIntegralStandard}
	\end{align}
	By passing through $L^p$ spaces so as to consider the adjoint and thus implicitly pass into negative Sobolev spaces, as in the proof of Theorem \ref{t:DeterministicOperator}, we find that there exist $\eta < 2$ and ${{C}}_{75,r,\beta}$ such that
	\[  \|e^{(t_1-\tau) \check\calL_\alpha} \nabla \check w_{t}^\eps \cdot \nabla\|_{\W{r} \to C^0} \| S_\eps(\tau,t_0) \|_{C^{2+\beta} \to \W{r}} \leq {{C}}_{75,r,\beta} \eps^2 (t_1 - \tau)^{-\eta/2}, \]
	which by integrating \R{eq:SolutionIntegralStandard} gives the necessary result.
\end{proof}

When $\rho$ has higher regularity, we have the following tighter result, comparable to Proposition \ref{p:DeterministicOperatorS4}:
\begin{proposition}\label{p:DeterministicOperatorS3Standard}
	Suppose $\rho \in \W{s}$ for $s > 3$. Then for all $\alpha \in [0,1]$, $\beta \in (0,1)$ there exists a constant $\check {{C}}_{98,\alpha,\beta,T}$ such that for all $t_0 \leq t_1$, $\eps \leq \eps_0$,
	\[\| \check S_{\eps,\alpha}(t_1,t_0) - e^{(t_1 - t_0) \check \calL_\alpha}\|_{C^{1+\beta} \to C^0} \leq \check {{C}}_{98,\alpha,\beta} (t_1 - t_0) \eps. \]
\end{proposition}

We now consider the variance error. The following proposition follows directly from Proposition \ref{p:ReciprocalZeta} and Theorem \ref{t:Operator}.
\begin{proposition}\label{p:KDEEstimate}
	There exists $Z_0$ such that if $\zeta = Z_0 \eps^{1/2}$, then for all $\eps \in [0,\eps_0]$,
	\begin{align} \| \rho_\eps \|_\zeta &\leq e^{2d Z_0} \rhonm \notag \\
	 \| \rho_\eps^{-1} \|_\zeta &\leq 2 \| \rho^{-1} \|_0 \notag \\
	\| \rho^M_\eps - \rho_\eps \|_\zeta &\leq e^{2d Z_0}\delta. \label{eq:KDEError}
	\end{align}
\end{proposition}

Using that $\check U_{\eps,\alpha} = \rho_\eps^{-\alpha}$ and another application of Proposition \ref{p:ReciprocalZeta} allows us to bound the weights:
\begin{proposition}\label{p:WeightZetaBoundsStandard}
There exists $Z_0$ such that if $\zeta = Z_0 \eps^{1/2}$, then there exists a constant ${{C}}_{76,}$ such that for all $\eps \in [0,\eps_0]$, $\alpha \in [0,1]$,
\begin{align*} \| \check U_{\eps,\alpha} \|_\zeta, \| 1/\check U_{\eps,\alpha} \|_\zeta, \| \check V_{\eps,\alpha} \|_\zeta, \| 1/\check V_{\eps,\alpha} \|_\zeta \leq {{C}}_{76,}.
\end{align*}
\end{proposition}

By combining these estimates with \R{eq:KDEError}, we obtain that
\begin{proposition}\label{p:WeightStandard}
	There exist constants $\check Z_0, \check {{C}}_{9}, \check {{C}}_{53}$ 
	such that for all $\eps \in [0,\eps_0]$, $\alpha \in [0,1]$, if $\delta \leq \check {{C}}_{9}$ then
	\[ \| \check U^M_{\eps,\alpha} - \check U_{\eps,\alpha}\|_\zeta,\,  \| \check V^M_{\eps,\alpha} - \check V_{\eps,\alpha}\|_\zeta,\, \| \check Y^M_{\eps,\alpha} - \check Y_{\eps,\alpha}\|_0,\, \| (\check Y^M_{\eps,\alpha})^{-1} - (\check Y_{\eps,\alpha})^{-1}\|_0 \leq \check {{C}}_{53} \delta, \]
	where $\zeta = \check Z_0 \eps^{1/2}$.
\end{proposition}

The next proposition then follows along the lines of Theorem \ref{t:WeightedOperatorConvergence}:
\begin{proposition}\label{p:WeightedOperatorConvergenceStandard}
	There exist $\check Z_0, \check {{C}}_{10}$ such that if $\check \zeta = Z_0 \eps^{1/2}$ and $\delta \leq \check {{C}}_{9}$ then for all $\eps \leq \eps_0$ and $n \in \mathbb{N}$,
	\[ \| \check\P^M_{\eps,\alpha} - \check\P_{\eps,\alpha} \|_\zeta,\, \| \check\J^M_{\eps,\alpha} - \check\J_{\eps,\alpha} \|_{0\to\zeta},\, \| \check\K^M_{\eps,\alpha} - \check\K_{\eps,\alpha} \|_{\zeta\to 0} \leq \check {{C}}_{10}\delta, \]
	and
	\[ \| \check\M^M_{\eps,\alpha,n} - \check\M_{\eps,\alpha,n} \|_{0} \leq \check {{C}}_{10} \delta n. \]
\end{proposition}

Using Propositions \ref{p:DeterministicOperatorStandard} and \ref{p:WeightedOperatorConvergenceStandard}, the proof of Theorem \ref{t:EigendataStandard} then follows by analogy with Theorem \ref{t:Eigendata}.

\appendix
\section{Proof of Theorem \ref{t:ASSA} and Proposition \ref{p:ASSAInitialisation}}\label{a:ASSAProof}
\begin{proof}[Proof of Theorem \ref{t:ASSA}]
	In this proof, we will find it useful to define the functions $l^{(n)} = \log U^{(n)} - \log U$, and similarly $l^{(n)}_a, l^{(n)}_b$.

	We begin by proving (a) using Birkhoff cones. Let $\Lambda^+$ be the set of positive, bounded functions on the support of $\mu$, and let $d_{\Lambda^+}$ be the projective cone Hilbert metric on $\Lambda^+/\mathbb{R}^+$
	\[ d_{\Lambda^+}(\phi,\psi) := \sup \log \tfrac{\phi}{\psi} - \inf \log \tfrac{\phi}{\psi} \leq 2\| \log \phi - \log \psi\|_{L^\infty}. 
	\]
	Then it is well-known \citep{Peyre19} that if
	\begin{equation} \theta := \tanh \left( \tfrac{1}{4} \sup_{x,y \in \supp \mu}
	d_{\Lambda^+}(\D\delta_{x},\D\delta_{y})\right) < 1, 
	\label{eq:ConeContractionRate}
	\end{equation}
	then by the Birkhoff cone theorem, for any $\phi \in \Lambda^+$
	\[  d_{\Lambda^+}(1/\D[\phi],U) = d_{\Lambda^+}(\D[\phi],\D[U]) \leq \theta d_{\Lambda^+}(\phi,U). \]
	This gives the contraction rate of standard Sinkhorn iteration.

	However, one can also check for any $\phi, \psi \in \Lambda^+$,
	\[  d_{\Lambda^+}(\sqrt{\phi \psi},U) \leq \half(d_{\Lambda^+}(\phi,U) + d_{\Lambda^+}(\psi,U)).\]
	Applying this to \RR{eq:ASSA1}{eq:ASSA3} gives that
	\[  d_{\Lambda^+}(U^{(n+1)},U) \leq \half(\theta^2 + \theta) d_{\Lambda^+}(U^{(n)},U).\]
	Finally, since
	\[ e^{l^{(n)}} = \sqrt{e^{l^{(n-1)}_a} / \P e^{l^{(n-1)}_a}},  \]
	where we recall that $e^{l^{(n-1)}_a} = U^{(n-1)}_a/U$, and furthermore since
	\[ \int (\P e^{l^{(n)}_a}- e^{l^{(n)}_a})\, \dd\mu = 0, \]
	we find that
	\[ \sup l^{(n)} \geq 0 \geq \inf l^{(n)} \]
	so
	\[ \|\log U^{(n)} - \log U\|_{L^\infty} \leq d(U^{(n)},U),\]
	giving us what we need.
	\\

	We now consider the local convergence rate. To use the spectral properties of the normalised operator $\P := U \D U$ we will pass to the $L^2(\mu)$ norm.

	Using that $\|l^{(0)}\|_{L^\infty} \leq k$, then by the previous part, for all $n$
	\[ \|l^{(n)}\|_{L^\infty} \leq 2k, \]
	and the same holds for $l^{(n)}_{a},l^{(n)}_{b}$. Consequently, taking logarithms of exponentials of these functions is Lipschitz with constant $e^{2k}$; furthermore, for any function $l$ with $\|l\|_{L^\infty} \leq 2k$,
	\[ \| e^l - l\|_{L^2} \leq k e^{2k} \|l\|_{L^2},  \]
	and so since $\|\P\|_{L^\infty} = \|\P\|_{L^2} = 1$,
	\[ \| e^{\P l} - \P e^{l} \|_{L^2} \leq \| e^{\P l} - \P l \|_{L^2} + \| \P (l - e^{l}) \|_{L^2}  \leq  2k e^{2k} \|l\|_{L^2}. \]
	Since $l\n_a = - \log(\P e^{l\n})$,
	\begin{align*}
	\| l\n_a + \P l\n \|_{L^2} &\leq e^{2k} \| e^{-l\n_a} - e^{-\P l\n}\|_{L^2} \\
	&= e^{2k} \| \P e^{-l\n} - e^{-\P l\n}\|_{L^2} \\
	&\leq k' \|l\n\|_{L^2},
	\end{align*}
	where $k' = ke^{4k}$. Similarly,
	\begin{align*}
	\| l\n_b - (\P)^2 l\n \|_{L^2}	&\leq k' \|l\n_a\|_{L^2} + \|l\n_a - \P l\n\|_{L^2} \leq k' (2+k') \|l\n\|_{L^2}
	\end{align*}
	Then, since $l^{(n+1)} = \half(l\n_a + l\n_b)$,
	\begin{align}
	\| l^{(n+1)} - \half \P (I - \P) l\n \|_{L^2} &\leq k''\| l\n\|_{L^2},
	\label{eq:lDiffError}
	\end{align}
	where $k'' = k' (2+\half k')$.

	Now, $\P$ is a Markov operator which is self-adjoint in $L^2(\mu)$ and, furthermore, positive semi-definite on this space as $\D$ is and $U$ is positive. 
	As a consequence, the spectrum of $\P$ is a subset of $[0,1]$. Hence, the spectrum of $\half \P (I-\P)$ is contained in $[-\tfrac{1}{8},0]$, and so its $L^2(\mu)$ norm is bounded by $\tfrac{1}{8}$. We thus have
	\[ \|l^{(n)}\|_{L^2} \leq \left(\tfrac{1}{8} + k' (2+\half k')\right)^n \| l^{(0)}\|_{L^2} \leq \left(\tfrac{1}{8} + k''\right)^n \| l^{(0)}\|_\infty.\]
%
	\\

	To prove part (c), we use \R{eq:lDiffError}, so that
	\begin{align*} \|l^{(n+1)} - l^{(n)}\|_{L^2} &\geq \| -(I + \half \P (1-\P))l\n \|_{L^2} - k'' \|l\n\|_{L^2}\\
	& \geq (\tfrac{7}{8} - k'')\|l\n\|_{L^2}\\
	& \geq \frac{\tfrac{7}{8} - k''}{\tfrac{1}{8}+k''}\|l^{(n+1)}\|_{L^2},\end{align*}
\end{proof}

\begin{proof}[Proof of Proposition \ref{p:ASSAInitialisation}]
	This proposition is a consequence of bounds in the rest of the paper. We have from Theorem \ref{t:Operator} that
	\[ \|\D^M_\eps 1 - \D_\eps 1\|_{L^\infty} \leq e^{2 d \zeta^2}\delta.\]
	It is a standard result on Gaussian kernels that
	\[ \|\D_\eps 1 - \rho \|_{L^\infty} = \|(\C_{\eps} - I) \rho\|_{L^\infty} \leq \half \eps \|\rho\|_{\W{2}}. \]
	Lemma \ref{l:Sinkhorn} gives us that if $\delta \leq {{C}}_{9}$ then
	\[ \| U^M_\eps - U_\eps\|_{L^\infty} \leq {{C}}_{53} \delta. \]
	As a result of Theorem \ref{t:UConvergence} and Lemma \ref{l:SinkhornDriftBound},
	\[ \| \log U_\eps - \log \rho^{-1/2} \|_{L^\infty} \leq {{C}}_{25,2}\eps. \]
	These results together mean that there exist constants ${{C}}_{11}, {{C}}_{12}$ such that if $\delta < {{C}}_{11}$ and $\eps < \eps_0$,
	\[ \|\log ((\D^M_\eps 1)^{-1/2}) - \log U^M_\eps \|_{L^\infty} \leq {{C}}_{12} (\delta + \eps), \]
	as required.
\end{proof}


\section{Proof of Proposition \ref{p:Compactness}}\label{a:CompactnessProof}
\begin{proof}[Proof of Proposition \ref{p:Compactness}]
	Let $z_\mathcal{E}$ be the centre of $\mathcal{E}$, set $\eta = \arcsinh(\zeta/\ell) < \eta_0$ and define $\To_\eta^d$ as the complex $\eta$-fattening of the hyper-torus $\To^d$:
	\[\To_\eta^d = \left((\mathbb{R} + i[-\eta,\eta])/2\pi\mathbb{Z}\right)^d \]
	Define the map $\tau: \To_\eta^d \to \mathcal{E}_\zeta$
	\[\tau(z) = \ell \cos z + z_\mathcal{E}. \]
	Now the Hardy space $H^\infty(\mathcal{E}_\zeta)$ is isometrically embedded in the Hardy space of bounded, even analytic functions on $\To_\eta^d$, $H^\infty_{\textrm{even}}(\To_\eta^d)$, via the map $C_\tau: \phi \mapsto \phi \circ \tau$. This map is also an isometric embedding of $C^0(\mathcal{E})$ into $C^0(\To^d)$.

	The Hardy space $H^\infty_{\textrm{even}}(\To_\eta^d)$ in turn is a subset of another Hardy space $H^2_{\textrm{even}}(\To_\eta^d)$, consisting of even analytic functions on $\To_\eta^d$ that are bounded with respect to the norm
	\[ \| \phi \|_{H^2_{\textrm{even}}(\To_\eta^d)}^2 = (4\pi)^{-d} \int_{\partial \To_\eta^d} |\phi(z)|^2 \dd z. \]
	Furthermore, $\| \phi \|_{H^2_{\textrm{even}}(\To_\eta^d)} \leq \| \phi \|_{H^\infty_{\textrm{even}}(\To_\eta^d)}$, so the image of the unit ball $C_\tau B_{H^\infty(\mathcal{E}_\zeta)}(0,1)$ is contained in the unit ball of $H^\infty_{\textrm{even}}(\To_\eta^d)$.

	On $H^2_{\textrm{even}}(\To_\eta^d)$ we have the compactness result:
	\begin{proposition}\label{p:TorusCompactness}
		The unit ball in $H^2_{\textrm{even}}(\To_\eta^d)$ may be covered by $C^0(\To^d)$ balls with centres a the finite set $\hat S_{\xi}^{\eta}$, and for $\eta \in (0, \eta_0)$ and $\xi \in (0,1/2)$ there exist constants ${{C}}_{83}, {{C}}_{84}$ depending only on $\eta_0, d$ such that
		\[ |\hat S_{\xi}^\eta| \leq e^{({{C}}_{83} \log \xi^{-1} + {{C}}_{84} \log \eta^{-1})(\eta^{-1} \log \xi^{-1})^{d}} \]
	\end{proposition}

	As a result, we can cover $B := C_\tau B_{H^\infty(\mathcal{E}_\zeta)}(0,1)$ by $C^0$ balls centred at the points in $\hat S_{\xi/2}^{\eta}$. It does not necessarily hold that $\hat S_{\xi/2}^{\eta} \subset B$, but because the diameter of a $\xi/2$-ball is bounded by $\xi$, around each $\xi/2$ ball that intersects $B$ we can choose a $\xi$-ball with a centres inside $B$. From the injectivity of the isometry $C_\tau$ we get the desired set $S_\xi^{\ell,\zeta}$.
\end{proof}

\begin{proof}[Proof of Proposition \ref{p:TorusCompactness}]
	The functions
	\[ b_{k}(z) = \prod_{j=1}^d \cosh{2k_j\eta}^{-1/2} \cos k_j z_j \]
	for $k \in \mathbb{N}^d$ form an orthonormal basis of $H^2_{\textrm{even}}(\To_\eta^d)$.

	Furthermore,
	\[ \|b_{k}(z)\|_{C^0(\To^d)} \leq \prod_{j=1}^d \cosh{2k_j\eta}^{-1/2} \leq 2^{d/2} e^{\sum k_j \eta}. \]

	Let us construct $\hat S_\xi^{\eta}$ so that a $\xi/2$-fattening of the subspace spanned by basis elements $\{b_k\}_{\sum k_j \leq k^*}$ for some $k^*$ covers the Hardy space ball, and then construct a lattice of functions inside this subspace.

	If we set
	\[k^* = \max\{d + 2\log(2^{2+d}\sqrt{d/2\pi} \xi^{-2})/\eta, 2 + 2d/\eta\} \leq {{C}}_{81} \eta^{-1} \log \xi^{-1},\]
	for some positive constant ${{C}}_{81}$ dependent only on $d, \eta_0$),
	we can choose
	\[ \hat S_\xi^\eta = \left\{ \sum_{\sum k_j \leq k^*} w_k b_k \mid w_k \in \left[-2^{-d/2} e^{-\sum k_j \eta},2^{-d/2} \eta^{-\sum k_j \eta}\right] \cap (\xi(k^*)^{-d/2}/2) \mathbb{Z} \right\}. \]

	A crude bound on the size of this set gives that
	\begin{align*}|\hat S_\xi^\eta| &\leq (2^{4-d/2} (k^*)^{d/2} \xi^{-1} + 1)^{(k^*)^d}\\
	&\leq
	({{C}}_{82} \xi^{-1}\eta^{-d/2} (\log \xi^{-1})^{d/2})^{({{C}}_{81} \eta^{-1} \log \xi)^d}\\
	&\leq e^{({{C}}_{83} \log \xi^{-1} + {{C}}_{84} \log \eta^{-1})(\eta^{-1} \log \xi^{-1})^{d}}
	\end{align*}
	for positive constants  (dependent only on $d, \eta_0$) ${{C}}_{82}, {{C}}_{83}, {{C}}_{84}$.
\end{proof}

\section{Proof of Proposition \ref{p:ReciprocalZeta} and Lemma \ref{l:Resolvent}}\label{a:ResolventProof}
\begin{proof}[Proof of Proposition \ref{p:ReciprocalZeta}]
	The second equation is a simple application of Proposition \ref{p:WeaktoStrong}.

	For the first equation we can say that
	\[ \psi(z)^{-1} =\int_{\mathbb{R}^d} g_\eps(z-y) \rho(y) \phi(y)\, \dd y \]
	and consequently,
	\begin{align*} \|\psi\|_\zeta^{-1} &= \inf_{z \in \domain_\zeta} \left| \int_{\mathbb{R}^d} g_\eps(z-y) \rho(y) \phi(y)\, \dd y \right| \\
	&= \inf_{x \in \domain, s \in [-\zeta,\zeta]^d} \left| \int_{\mathbb{R}^d} e^{(|s|^2/2 + i s\cdot(x-y))/\eps} g_\eps(x-y) \rho(y)\phi(y)\, \dd y \right|.
	\end{align*}
	Using that $\phi > 0$ on the real domain $\domain$ and that $|e^{i w} - 1| \leq |w|$ for real $w$, we have then that
	\[ \|\psi\|_\zeta^{-1} \geq \inf_{x \in \domain, s \in [-\zeta,\zeta]^d} e^{|s|^2/2\eps}\left((\D_\eps \phi)(x)- \left|\int_{\mathbb{R}^d} g_\eps(x-y) \eps^{-1} (s \cdot (x-y)) \rho(y) \phi(y)\, \dd y \right| \right).\]
	Using that
	\[\| g_\eps(x) s\cdot x\|_{L^1(\mathbb{R}^d,\dd x)} \leq \sqrt{2\eps/\pi} |s| \leq \sqrt{2\eps/\pi} \sqrt{d \zeta^2} \leq \sqrt{2dZ_0/\pi}\eps,\]
	we have
	\[ \|\psi\|_\zeta^{-1} \geq \|\D_\eps \phi\|_0 - \sqrt{2dZ_0/\pi} \|\rho\|_0 \|\phi\|_0.\]
	Because
	\[ \| \D_\eps \phi \|_0^{-1} \leq \|\rho^{-1}\|_0 \|\phi^{-1}\|_0, \]
	our assumption on $Z_0$ gives us the required bound.
\end{proof}

\begin{proof}[Proof of Lemma \ref{l:Resolvent}]
	Consider the following forward equation on the domain $\domain$ for $t \in [0,\eps]$:
	\begin{equation} \partial_t \phi^t = \calL \phi^t + \nabla \hat w^t_\eps \cdot \nabla \phi^t, \label{eq:TransitionSDE} \end{equation}
	recalling that
	\begin{equation} e^{\hat w^t_\eps} = \rho^{-1} e^{\half t\Delta} (U_\eps \rho). \label{eq:hatwtEvolution} \end{equation}
	We have from Theorem \ref{t:DeterministicOperator} that $\P_\eps^m$ is given by
	\[ \P_\eps\phi = S_{\eps}(m\eps,0) \phi(y)\, \dd y, \]
	where the solution operator $S_{\eps}(t_1,t_0)$ is a kernel operator.
	We can thus use PDE results to study the functional behaviour of $\P_\eps$.

	We divide the operator
	\begin{equation} (I + \P_\eps)^{-1} = I - \P_\eps + (I - \P_\eps^{2n^*})^{-1} \sum_{n=2}^{2n^*+1} (-1)^{n} \P_\eps^{n}, \label{eq:ResolventDivision} \end{equation}
	where $n^* = \lceil (2\eps^{-1}) \rceil$, and consider in turn the norms of $(I - \P_\eps^{2n^*})^{-1}$ and $\P_\eps^{2n+2}-\P_\eps^{2n+1}$.
	\\

	We have from Corollary \ref{c:WeightBound} that
	\[\sup_t \|\log \rho + \hat w^t_\eps\|_0 \leq \|\log \rho\|_0 + {{C}}_{23,0},\]
	and consequently Gaussian lower estimates on the fundamental solution from Theorem 1 of \citet{Liskevich00}\footnote{Here as usual we use that we can extend $\domain = (\mathbb{R}/L\mathbb{Z})^d$ to $\mathbb{R}^d$ in the natural way.} imply that there exists a constant ${{C}}_{85} \in (0,1)$ depending on $L, d, {{C}}_{22}, \rhonm, \Lip \log \rho,\eps_0$ such that for all bounded non-negative functions $\phi$,
	\[ \inf S(2n^*\eps,0)\phi \geq {{C}}_{85} \|\phi\|_{L^\infty}, \]
	where we recall that $2n^*\eps \in [1,1+2\eps_0]$.
	The Sinkhorn balancing \R{eq:SinkhornProblem} makes $\P_\eps$ bistochastic, so $\|\P_\eps\|_{0} = 1$: if $\|\phi\|_{0} = 1$ and $\int_\domain \phi \,\dd x = 0$ then
	\[ \P_\eps^{2n^*} \phi = \P_\eps^{2n^*} \phi^+ - \P_\eps^{2n^*} \phi^- = (\P_\eps^{2n^*} \phi^+ - {{C}}_{85}) - (\P_\eps^{2n^*} \phi^- - {{C}}_{85}), \]
	where $\phi^+,\phi^- \geq 0$ are the positive and negative parts of $\phi$ respectively.

	Since the two bracketed quantities are non-negative, we have
	\[ \|\P_\eps^{2n^*} \phi\|_{0} \leq \max\left\{\sup \P_\eps^{2n^*} \phi^+ - {{C}}_{85},\, \sup \P_\eps^{2n^*} \phi^- - {{C}}_{85} \right\} = \|\P_\eps \phi\|_0 - {{C}}_{85} = 1 - {{C}}_{85}. \]
	Thus,
	\[ \| \P_\eps^{2n^*} \|_0 \leq 1 - {{C}}_{85} < 1 \]
	and so
	\[ \| (I- \P_\eps^{2n^*})^{-1} \|_0 \leq {{C}}_{85}^{-1}. \]
	\\

	On the other hand, we have the Schauder estimate from Theorem 1 of \citet{Knerr80} that there exists a constant ${{C}}_{86}$ depending on $d, L, \eps_0, {{C}}_{22}$ such that for $0 < t_0 < t_1 < 1+2\eps_0$ and $\mathcal{A} \in \{\Delta, \nabla\}$,
	\begin{align*}
	\left\| \mathcal{A} S_{\eps}(t_1,0) \phi - \mathcal{A} S_{\eps}(t_0,0) \phi \right\|_0 &\leq {{C}}_{86} s^{-(1+\beta/2)} (t_1-t_0)^{\beta/2} \|\phi\|_0.
	\end{align*}
	Since $\hat w^t_\eps$ is $\eps$-periodic, we can apply these equations with the evolution of $S_\eps$ \R{eq:TransitionSDE} to say that for $t \in [0,1+\eps_0]$,
	\[ \left\| \frac{\partial}{\partial t} S_\eps(t+\eps,0) \phi - \frac{\partial}{\partial t} S_{\eps}(t,0) \phi \right\|_0 \leq 2{{C}}_{86} t^{-(1+\beta/2)} \eps^{\beta/2}  \|\phi\|_0.  \]

	As a result,
	\begin{align*} \| \half(\P_\eps^{2n+2}-2\P_\eps^{2n+1} + \P_\eps^{2n}) \| &= \left\| \frac{1}{2} \int_{2n\eps}^{(2n+1)\eps} \left(\frac{\partial}{\partial t} S_\eps(t+\eps,0) - \frac{\partial}{\partial t} S_{\eps}(t,0)\right)\,\dd t \right\|\\
	& \leq {{C}}_{86} (2n\eps)^{-(1+\beta/2)} \eps^{1+\beta/2} \\
	&= {{C}}_{86} (2n)^{-(1+\beta/2)} \end{align*}

	Since, recalling \R{eq:ResolventDivision},
	\[ \sum_{n=2}^{2n^*+1} (-1)^{n} \P_\eps^{n} = \half (-\P_\eps^{2n^*+2} + \P_\eps^{2}) + \sum_{n=1}^{n^*} \half (\P_\eps^{2n+2} - 2\P_\eps^{2n+1} + \P_\eps^{2n}), \]
	we have
	\[ \left\|\sum_{n=2}^{2n^*+1} (-1)^{n} \P_\eps^{n} \right\|_0 \leq 1 + {{C}}_{86} 2^{-(1+\beta/2)} (1 + 2/\beta) \]
	and so
		\[ \| (I + \P_\eps)^{-1} \|_0 \leq 2 + {{C}}_{85}^{-1} \left(1 + {{C}}_{86} 2^{-(1+\beta/2)} (1 + 2/\beta)\right) =: {{C}}_{52} \]
		as required.
\end{proof}

\medskip

\noindent
{\bf Acknowledgement.} This research has been partially funded by
Deutsche Forschungsgemeinschaft (DFG, German Science Foundation) - SFB 1294/1 - 318763901.

This research has also been supported by the European Research Council (ERC) under the European Union's Horizon 2020 research and innovation programme (grant agreement No 787304)


CLW would like to thank Jakiw Pidstrigach, Daniel Daners, David Lee and Harry Crimmins for helpful discussions.

%
%

\bibliography{diffmaps}

\end{document}